 \documentclass[reqno,11pt]{amsart}
\usepackage{a4,latexsym,amssymb,amsfonts}

\usepackage{amsmath,amscd} 
\usepackage{pdfsync,
verbatim} 

\usepackage{enumerate}
\usepackage{amsthm}
\usepackage{color}

\numberwithin{equation}{section}

\newtheorem{theorem}{Theorem}[section]
\newtheorem{lemma}[theorem]{Lemma}
\newtheorem{proposition}[theorem]{Proposition}
\newtheorem{corollary}[theorem]{Corollary}

\theoremstyle{definition}
\newtheorem{remark}[theorem]{Remark}

\begin{document}

\title[Strichartz estimates for the sublaplacian]{Strichartz estimates 
for the Schr\"odinger equation  for the sublaplacian
on complex spheres }
\author{Valentina Casarino and Marco M. Peloso}
\address{DTG\\
Universit\`a degli Studi di Padova\\
Stradella san Nicola 3\\
I-36100 Vicenza}
\email{valentina.casarino@unipd.it}
\address{Dipartimento di Matematica\\
Universit\`a degli Studi di Milano\\
Via C. Saldini 50\\
I-20133 Milano}
\email{marco.peloso@unimi.it}
\thanks{}
\date{\today}

\maketitle

\begin{abstract}
In this paper we consider the sublaplacian $\mathcal L$
on the unit complex sphere  $S^{2n+1}\subset {\mathbf{C}}^{n+1}$,
equipped with its natural CR structure,
and derive
Strichartz estimates 
with fractional loss of derivatives for the solutions of the free
Schr\"odinger
equation associated with $\mathcal L$.
Our results are stated in terms of certain Sobolev-type spaces,
that measure the
regularity of functions on $S^{2n+1}$   differently according to their spectral
localization.
Stronger conclusions are obtained for particular classes of solutions, 
corresponding to  initial data whose spectrum 
is contained in a proper cone of ${\mathbf{N}}^2$.

\end{abstract}

\bigskip
\section{Introduction and statement of the main results}\medskip
In the last two decades the dispersive properties 
of evolution equations
have been extensively investigated 
and successfully applied in different contexts, such as the local and
global existence for non-linear equations, the well-posedness  theory 
in Sobolev spaces and the scattering theory. 
Dispersive and smoothing properties are now essentially well understood 
for the Schr\"odinger equation in the Euclidean setting, where both
$L^p-L^q$
bounds and  Strichartz estimates 
have been proved for a wide class of linear and non-linear problems.
 
In this paper, we
study the dispersive properties 
of the free  Schr\"odinger equation associated with the sublaplacian $\mathcal L$ on 
the unit
complex sphere $S^{2n+1}$  in ${\mathbf{C}}^{n+1}$, $n\ge 1$,
\begin{equation}
\label{schreq}
\begin{cases} 
i\partial_{t} v+\mathcal L v= 0
\cr
v(0,z)=v_0\,,
\end{cases}
\end{equation}
where $v_0\in L^2 (S^{2n+1})$ and
$\mathcal L$
denotes the sublaplacian, that is,
the operator  defined by 
\begin{equation}\label{sublap}
\mathcal L:= -\sum_{1\le j<k\le n+1}M_{jk}\overline{M}_{jk} + \overline{M}_{jk}M_{jk}\, ,
\end{equation}
with
$M_{jk}:=\overline{z}_j
\partial_{z_k}-\overline{z}_k \partial_{z_j}$.
The operator $\mathcal L$ is a densely defined, self-adjoint, positive,
and subelliptic operator on $S^{2n+1}$ \cite{Geller} and
it 
coincides with  the real part of the Kohn--Laplacian acting on
functions \cite{Lee}; see also \cite{MPR}. 
The sublaplacian $\mathcal L$ 
 may be considered as the subriemannian analogue
of the Laplace--Beltrami operator on a 
Riemannian manifold, see e.g. \cite{Jer-Lee}. \medskip

Our main result is a Strichartz estimate 
for the solution $v$ of \eqref{schreq}.

Strichartz estimates
are  a family of space-time bounds on solutions of \eqref{schreq}, 
which provide  a useful tool to control the norm of the solutions.
In particular, using the notation
$L^p_t\,L^q_x:=
L^p (I_t, L^q({\mathbf{R}}^n_x))$,
we bound the $L^p_t \,L^q_x$ norm
of  $v$ 
by means of a suitable mixed Sobolev norm, denoted by
$\|v_0\|_{{\mathcal X}^{(r,s)}}$, of the initial datum.  

In order to describe the mixed Sobolev spaces ${{\mathcal X}^{(r,s)}}$, we
start from  the classical 
 decomposition of the space of square integrable functions on
$S^{2n+1}$
\begin{equation}\label{dec-L2}
L^2 (S^{2n+1})=\bigoplus_{\ell,\ell'=0}^{\infty}
\mathcal{H}^{\ell,\ell'}
\, ,
\end{equation}
$\mathcal{H}^{\ell,\ell'}$ being the space of complex spherical harmonics of bidegree
$(\ell,\ell')$, \cite[Ch.~11]{ViK}.

This decomposition
is the joint spectral decomposition of $\mathcal L$ and the Laplace--Beltrami
operator $\Delta$ on $S^{2n+1}$, since the subspaces
$\mathcal{H}^{\ell,\ell'}$ are  eigenspaces both for $\Delta$ with eigenvalue 
$\mu_{\ell,\ell'}= (\ell+\ell')(\ell+\ell'+2n)$,
and for $\mathcal L$ with  eigenvalue
$\lambda_{\ell,\ell'}= 2\ell\ell'+n(\ell+\ell')$.

Now fix $M>1$ and define
\begin{equation}\label{cV}
\mathcal V=\big\{(\ell,\ell')\in{\mathbf{N}}^2:\, \ell/M <\ell'< M\ell\big\}\, ,
\end{equation}
and $\mathcal E$ to be its complementary region in ${\mathbf{N}}^2$.
Observe that
when $(\ell,\ell')\in\mathcal V$
then 
$\mu_{\ell,\ell'}\approx\lambda_{\ell,\ell'}$,
while if
$(\ell,\ell')\in\mathcal E$, that is, 
if $\ell'\le \ell/M$
or 
$M\ell\le\ell'$, then
$\mu_{\ell,\ell'}$ grows as $\max(\ell,\ell')^2$, while
$\lambda_{\ell,\ell'}$ varies between $\max(\ell,\ell')$ and $\max(\ell,\ell')^2$. 

Hence,
we are led to introduce appropriate Sobolev-type spaces
that measure the
regularity of functions differently according to their spectral
localization.  For $r\ge0$ we denote by
$W^r (S^{2n+1})$  the standard non-isotropic Sobolev space, for
instance 
defined as the image of $L^2 (S^{2n+1})$ under $(I+\mathcal L)^{-r/2}$.

We now define 
 $\mathcal X^{(r,s)}_M(S^{2n+1})$ 
as the space of all functions 
$u\in L^2 (S^{2n+1} )$, 
spectrally decomposed as 
$u=\sum_{\ell,\ell'=0}^{\infty}u_{\ell,\ell'}$, $u_{\ell,\ell'}\in\mathcal{H}^{\ell,\ell'}$,
such that 
$$ \sum_{\ell/M < \ell' < M\ell  } u_{\ell,\ell'} \in W^r (S^{2n+1})\, ,
$$
while the complementary sums 
$$ 
\sum_{ \ell'\le \ell/M\,  } u_{\ell,\ell'}\,, \sum_{\ell'\ge M \ell} u_{\ell,\ell'} \in W^s (S^{2n+1})\,.
$$

It may be seen as a natural fact that we need to consider a
two-parameter scale for the Sobolev spaces, since
\eqref{dec-L2} is a two-indices decomposition of $L^2(S^{2n+1})$.

The Strichartz estimates that we are able to prove for
solutions of \eqref{schreq}
are expressed in terms of $\mathcal X^{(r,s)}_M$-norms
and are derived as a consequence of some kind of dispersive estimates.
We denote by $Q:=2n+2$ the {\em homogeneous dimension} of $S^{2n+1}$
(see also Section \ref{prelim}).

\begin{theorem}\label{strichartz}
Let $S^{2n+1}$ denote the unit complex sphere in ${\mathbf{C}}^{n+1}$
and let $\mathcal L$ be the sublaplacian, defined by
\eqref{sublap}.
Let $p\ge 2$, 
$q<+\infty$
satisfy the admissibility condition
\begin{equation}\label{condammiss}
\frac{2}{p}+\frac{Q}{q}=\frac{Q}{2}\,.
\end{equation} 
Define
\begin{equation}\label{def-sn}
s_n:=\begin{cases}
 2[1-1/(n+1)]\,, &\text{if $n>1$}\cr
4/3\,,&\text{if $n=1$}\,.\cr
 \end{cases}
\end{equation}
Let $M>1$ be fixed.  Then, if
$I$ is any finite time interval and $s\ge s_1$ or $s>s_n$ for $n>1$,
there exists a constant $C=C(s,I,M)>0$ such that any solution $v$ of \eqref{schreq} satisfies
the estimate 
\begin{equation}\label{Strichartz}
\|v\|_{L^p(I, L^q (S^{2n+1}))}\le C \, \|v_0\|_{\mathcal X^{(s/p, 2/p)}_M(S^{2n+1})}
\, .
\end{equation}
\end{theorem}
\medskip

While there exists a vast literature on 
Strichartz estimates and their application to
the non-linear Schr\"odinger equation
for the Laplace--Beltrami operator on  Riemannian manifolds (see the
comments below), little is known in the case of the sublaplacian on CR
manifolds, even in the case of the Heisenberg group ${\mathbf{H}}_n$.
 Indeed,
since the Heisenberg group ${\mathbf{H}}_n$
is biholomorphically equivalent to the unit sphere $S^{2n+1}$
with the north pole removed via the Cayley transform, 
Theorem \ref{strichartz} should be in particular compared with results concerning the 
Schr\"odinger equation on 
${\mathbf{H}}_n$. 

 H. Bahouri, P. G\'erard and C.-J. Xu
in the seminal paper \cite{BahouriGerardXu} prove that 
no global in time dispersive estimate may hold for solutions 
of the Schr\"odinger equation on ${\mathbf{H}}_n$; see also the more recent
work by  G\'erard and S. Grellier \cite{GG1,GG2}. 
However,
the same lack of dispersion occurs
on  $S^{2n+1}$, and, in
addition, no local in time dispersive estimate
can hold, as we shall observe in Section \ref{dispersiveprima}.
Nonetheless, we are able to prove local  Strichartz estimates for the solutions of
\eqref{schreq},  by substituting the 
dispersive estimate for the Schr\"odinger
propagator $e^{it\mathcal L}$ by
a family of dispersive estimates
for the  frequency localized operator 
$e^{it\mathcal L}\varphi (h^2 \mathcal L)$.
This idea
originally appeared 
in \cite{Bahouri}
and
\cite{Tataru}, and has been successfully applied in  the work of 
Burq, G\'erard and
Tzvetkov \cite{Burq1,Burq2}. 
 In Theorem \ref{dispersive} we
 prove such  spectrally localized
dispersive estimates,
by means of a careful analysis of  the oscillation of 
the infinite sum,
depending on the two indices $\ell$ and $\ell'$,
 that defines the integral kernel of 
the operator
$e^{it\mathcal L}\varphi (h^2 \mathcal L)$.
The proof of Theorem \ref{dispersive} is quite delicate and occupies a
good portion of the present paper.
One might wonder if the same technique could apply to the Heisenberg
framework and 
this topic will be the object of further investigation.
\medskip

It is interesting to compare our results with the known ones in the
Riemannian framework. 
 Consider a Riemannian manifold $(\mathcal M,g)$ of dimension $d$ and
 the
Schr\"odinger equation
\begin{equation}
\label{schreqRiem}
\begin{cases} 
i\partial_{t} u+\Delta_g u=0
\cr
u(0,x)=u_0\,,
\end{cases}
\end{equation}
where $\Delta_g$ denotes  the Laplace--Beltrami operator on $(\mathcal M,g)$.
Then  Strichartz estimates of solutions of 
\eqref{schreqRiem}
are usually of the
form
\begin{equation}\label{Strichartzgenerale}
\|u\|_{L^p ([-T,T], L^q(\mathcal M))}\le C
\| u_0\|_{H^s (\mathcal M)}\,,
\end{equation}
where $p, q, d$ satisy the scale-invariance  
condition
\begin{equation}\label{condammissriem}
\frac{2}{p}+\frac{d}{q}=\frac{d}{2}\,.
\end{equation} 
 Here and in what follows, we denote by $H^s$  the classical
Sobolev space on $\mathcal M$, which may be  defined as the image of 
$L^2 (\mathcal M)$ under $(I+\Delta_g)^{-s/2}$.
The key ingredient to prove \eqref{Strichartzgenerale}
is given by some dispersive
estimates, that is, estimates of the $L^\infty$ norm of solutions of 
\eqref{schreqRiem}
at a fixed time $t$.
 
When $\mathcal M={\mathbf{R}}^n$,
the theory is basically well-established 
and one can choose $s=0$ and $T=\infty$
in \eqref{Strichartzgenerale}, 
thanks to the essential contributions by
Strichartz, 
Ginibre and Velo, and
Keel and Tao \cite{Strichartz, GinibreVelo, KeelTao}.

When $\mathcal M$ is a generic Riemannian
manifold, the situation is more involved
and  the geometry, as  it is well known, plays an essential r\^ole.

On compact manifolds  the dispersive effect is generally  weak;
nonetheless, 
Burq, G\'erard and Tzvetkov,
generalizing the earlier work of J. Bourgain
on tori \cite{Bou1,Bou2}, 
proved 
an estimate like 
\eqref{Strichartzgenerale}
on any compact and boundaryless manifold $\mathcal M$, 
with $s=1/p$, with a loss of derivatives with respect to the flat 
Euclidean case, but again with a gain of $1/p$ 
derivatives in comparison  to the bounds indicated by Sobolev embeddings \cite{Burq2}.
Later, Blair, Smith and Sogge proved Strichartz estimates 
with $s=4/3p$  both
for a compact Riemannian manifold with boundary and for a compact manifold $\mathcal M$
without boundary, endowed with a Lipschitz metric $g$ \cite{Blair}
(see also \cite{Blair2}, where these results have been recently improved).
\medskip

Then the spirit of Theorem \ref{strichartz} is that, if the initial datum $v_0$ 
is spectrally localized in a proper angular sector $\mathcal V$ in ${\mathbf{N}}^2$ defined as
in \eqref{cV},
then
we are able to prove a Strichartz estimate
like 
\eqref{Strichartzgenerale}, where the 
$H^s$ norm of the initial datum at the right-hand side is
replaced by 
the standard  non-isotropic norm $W^{s/p}$,
 for any index $s$ such that  $s> 2 [1-1/(n+1)] $ if $n>1$ or $s\ge
 4/3$ if $n=1$. Thus  
there is  a gain of $2 /(n+1)p$ 
derivatives ($2/3p$ in the one dimensional case) in comparison  to the
bounds indicated by non-isotropic Sobolev embeddings. 
If the initial datum $v_0$ 
is spectrally localized in 
the complementary region, that is, for instance, 
if
$v_0= \sum_{0\le \ell'\le \ell/M}h_{\ell,\ell'}$, $h_{\ell,\ell'}\in\mathcal{H}^{\ell,\ell'}$, then
our techniques only lead to  an estimate
like
\eqref{Strichartzgenerale}, with
$H^s$ norm of the initial datum at the right-hand side replaced by
the standard  non-isotropic norm ${W^{2/p}}$, thus providing no
improvement with respect to 
the Sobolev embedding. 
\medskip

The problem of optimality 
for Strichartz estimates is in general open, also in the Riemannian set-up.
As it is well known, the Strichartz estimate proved in  \cite{Burq2}
is not sharp, unless in the case $p=2$, 
in the class of compact Riemannian manifolds, since
J. Bourgain proved that for the flat torus $({\mathbf{R}}/2\pi{\mathbf{Z}})^2$
the Strichartz estimate holds for $p=q=2$
with  loss of $\varepsilon$ derivatives, for every $\varepsilon>0$ \cite{Bou1,Bou2}.
Moreover, 
Burq, G\'erard and Tzvetkov were able to improve their intermediate 
Strichartz estimates 
in some specific geometries, like spheres and
Zoll surfaces \cite{Burq3, Burq4}.
In our framework, 
some sharp bounds for the eigenfunctions of the sublaplacian on
$S^{2n+1}$, recently proved by the first author \cite{Casarino1,Casarino2}, 
do not suffice to prove the optimality 
in \eqref{Strichartz}.

It is worth noticing that different approaches, which have been
succesfully used in the Riemannian context 
(we refer in particular to \cite{Burq3,Burq2}), are possible and could
be used to prove optimal bounds, at least for intermediate $(p,q)$;  
in particular, 
it would be interesting to 
prove 
multilinear  estimates  for spectral projections 
associated to $\mathcal L$ on the complex sphere,
as well as  to prove intermediate Strichartz estimates by following
 the Fourier analytic approach by Bourgain.

We would also like to point out  that the compact manifold  $S^{2n+1}$, beyond
the pioneering works of G. Folland and D. Geller \cite{Follanddbar,Geller},
has recently attracted a lot of interest in connection with its CR
structure; 
we refer, 
in particular, to the recent papers
\cite{Branson,BW,Cowling-and-others}
  and to \cite{Casarino1,Casarino2,CaPe}. 
\medskip

The  paper is organized as follows.
In Section \ref{prelim}
we start recalling the basic facts about harmonic analysis on the
complex sphere. Then we recall the definition of the standard
isotropic and non-isotropic  
Sobolev spaces on the sphere
and introduce the mixed Sobolev spaces.
In Sections
\ref{dispersiveprima}
and
\ref{Steps3-4} we prove
the basic
dispersive estimates for solutions of
\eqref{schreq}
localized at high frequencies.
Our proof hinges on a repeated use of the Poisson summation formula,
specifically adapted in the key Lemma \ref{Poisson}
to our case. 
Following a classical approach, we then deduce in Section \ref{dim_teo}
  the Strichartz
 estimate
\eqref{Strichartz}
from the dispersive bounds.
 Optimality  will be discussed in Section \ref{finalremarks}, where we
 also make a comment on  
other possible admissibility conditions.

We shall use the symbol 
 $C$ to denote constants which may vary from one formula to the 
next, and $\lfloor x\rfloor$ to denote the greatest integer
at most $x$.
The symbol $\approx$ between two positive expressions means that their
ratio is bounded above and below. 
\medskip

\section{Preliminary Facts and Notation}\label{prelim}
In this  section we recall some basic facts about
spherical harmonics and their relation to the the  analysis on the
complex sphere. 
 
For $n\ge 1$, we denote by
${\mathbf{C}}^{n+1}$  the $(n+1)$-dimensional complex space
equipped with the scalar product
$\langle z,w\rangle:=z_1 \bar w_1 +\cdots+z_{n+1}\bar w_{n+1}$,  $z,w\in{\mathbf{C}}^{n+1}$, and
by
 $S^{2n+1}$   the unit sphere in
${\mathbf{C}}^{n+1}$ 
$$ 
S^{2n+1} 
=\bigl\{z=(z_1,\dots,z_{n+1})\in{\mathbf{C}}^{n+1}:\, \langle z,z\rangle=1 \bigr\}\,.
$$

The sphere $S^{2n+1}$ is a  strongly pseudoconvex CR manifold
and thus 
endowed with subriemannian structure.
The Carnot--Carath\'eodory distance associated with the operator $\mathcal L$ is
equivalent to the so-called {\em Kor\'anyi distance} $d$
\begin{equation}\label{distanza}
d(z,w):= \left| 1-\langle z,w\rangle \right|^{1/2}\,,
\end{equation}
 $z,w\in S^{2n+1}$,
see \cite{Nagel}.

The
{\it homogeneous dimension}
$Q$ of $S^{2n+1}$, that will play a relevant r\^ole in our analysis,
is given by 
$Q:=2n+2$, since
 it is well known  that 
 $\operatorname{Vol} (B(z,r))\sim r^{Q}$, where $B (z,r)$
denotes the  ball centered at $z\in S^{2n+1}$ with radius $r>0$.

\subsection{\it{Spherical harmonics and spectral projections}} 
Consider the space $L^2(S^{2n+1})$,
equipped with the inner product
$$
(f,g):= \int_{S^{2n+1}}
f(z) \overline{g(z)} \, d\sigma (z)\, ,
$$
where $d\sigma$ is the Lebesgue surface measure, which is invariant
under the action of the unitary group $U(n+1)$.

For non-negative integers $\ell,\ell'$,
${\mathcal{H}}^{\ell,\ell'}$ is the  vector space of the restrictions to
$S^{2n+1}$  of harmonic polynomials
$p(z,\bar{z})$, homogeneous of degree $\ell$ in
$z$ and of  degree $\ell'$ in $\bar z$.
A  function in ${\mathcal{H}}^{\ell,\ell'}$ is called
a  {\it{complex spherical harmonic of bidegree}} $(\ell,\ell')$.

When $\ell'=0$, 
the space  ${\mathcal{H}}^{\ell,0}$ consists of holomorphic polynomials, and
${\mathcal{H}}^{0,\ell}$ consists of polynomials whose complex conjugates are
holomorphic.

The subspaces ${\mathcal{H}}^{\ell,\ell'}$ have finite dimension $d_{\ell,\ell'}$ given by 
\begin{equation}\label{dimensione}
d_{\ell,\ell'}:=n \frac{\ell+\ell'+n}{\ell\ell'}
\begin{pmatrix} \ell+n-1\\
	 \ell-1\end{pmatrix}
\begin{pmatrix} \ell'+n-1\\
	 \ell'-1\end{pmatrix}
\end{equation}
if $\ell,\ell'\ge 1$, and by
$$
d_{\ell,0}=d_{0,\ell}:=\begin{pmatrix} \ell+n\\
  \ell\end{pmatrix}\, ,
$$
if $\ell$ or $\ell'$ equals $0$.  

Moreover, 
the subspaces ${\mathcal{H}}^{\ell,\ell'}$ 
are $U(n+1)$-invariant,
pairwise orthogonal and their sum is dense in $L^2(S^{2n+1})$; 
more explicitly, 
if we denote
by the symbol $\pi_{\ell,\ell'}$ the orthogonal projector mapping $L^2(S^{2n+1})$
 onto ${\mathcal{H}}^{\ell,\ell'}$, then 
each function $f\in L^2 (S^{2n+1})$ may be decomposed in a unique way as 
\begin{equation}\label{decomposition}
f= 
\sum_{\ell,\ell'=0}^{+\infty}  \pi_{\ell,\ell'}f\,,
\end{equation}
where the series converges unconditionally to $f$ in the $L^2$-topology.

A special r\^ole in ${\mathcal{H}}^{\ell,\ell'}$ is played  by the so-called 
{\em zonal}
functions.
Let $\big\{Y_k^{\ell,\ell'}\big\}$, $k=1,\dots,d_{\ell,\ell'}$,
be an orthonormal basis for ${\mathcal{H}}^{\ell,\ell'}$.  
For $(z,w) \in S^{2n+1}\times S^{2n+1}$ 
set
$$
Z_{\ell,\ell'} (z,w) :=\sum_{k=1}^{d_{\ell,\ell'}}
Y_k^{\ell,\ell'} (z)
\overline{Y_k^{\ell,\ell'}(w)}
\, .
$$

Then, for all $f\in{\mathcal{H}}^{\ell,\ell'}$ we have
\begin{equation}\label{pairing}
f(z)=\int_{S^{2n+1}} f(w) Z_{\ell,\ell'}(z,w)\, d\sigma(w)\, .
\end{equation}
Since ${\mathcal{H}}^{\ell,\ell'}$ is finite dimensional, the above pairing makes sense for
all $f\in L^2(S^{2n+1})$.

For each fixed point $w\in S^{2n+1}$, the function
$f(w)=Z_{\ell,\ell'}(\cdot,w)$ is in ${\mathcal{H}}^{\ell,\ell'}$ and it is constant on the orbits of the
stabilizer of $w$ in $U(n+1)$, which is isomorphic to $U(n)$.  In
other words $Z_{\ell,\ell'} (z,w)$ depends only on $\langle z,w\rangle$, and we
write
\begin{equation}\label{cos-eitheta}
\langle z,w\rangle =   e^{i\omega} \cos \theta\, ,\qquad \theta\in[0,\pi/2]\, , \
\omega\in[0,2\pi)\, .
\end{equation}
With an abuse of notation, we will also denote by $Z_{\ell,\ell'}$ the function 
depending on the $1$-dimensional complex variable $\langle z,w\rangle$, that
is,
$$
Z_{\ell,\ell'}\big(\langle z,w\rangle\big)=Z_{\ell,\ell'}(z,w)\, .
$$ 
 
An explicit formula for the zonal function $Z_{\ell,\ell'}\in{\mathcal{H}}^{\ell,\ell'}$, for
$\ell'\ge\ell\ge 1$, is given by
\begin{equation}\label{zonali}
Z_{\ell,\ell'} (e^{i\omega}\cos\theta )= 
 \frac{d_{\ell,\ell'}}{\omega_{2n+1}}
\frac{\ell! (n-1)!}{(\ell+n-1)!}
e^{i\omega (\ell'-\ell)} (\cos\theta)^{\ell'-\ell}  P_{\ell}^{(n-1, \ell'-\ell)}
 (\cos 2\theta) \,,
\end{equation}
where $\omega_{2n+1}$ denotes the surface area of
$S^{2n+1}$
and 
$P_{\ell'}^{(n-1,\ell-\ell')}$ is the Jacobi polynomial, see
\cite{Szego}.

For the case $\ell'<\ell$, it suffices to recall that $Z_{\ell,\ell'}(z,w)
=\overline{Z_{\ell',\ell}(w,z)}$.\medskip

Since $P_0^{(n-1,\ell)}\equiv1$, if $\ell'=0$ 
the zonal function is given by
$$
Z_{\ell,0} (z,w)=\frac{1}{\omega_{2n+1}}
\binom{\ell+n}{\ell}  \overline{\langle z,w\rangle }^\ell \, .
$$

The following bound for the zonal functions is well known,
and appears 
in \cite{FollandPoisson}.
For any $z,w\in S^{2n+1}$ 
we have
\begin{equation}\label{2.7}
|Z_{\ell,\ell'}(z,w)|\le
\frac{d_{\ell,\ell'}}{\omega_{2n+1}}\, .
\end{equation}

Finally, it is easy to check that 
the orthogonal projector $\pi_{\ell,\ell'}$  
may be written as 
$$
\pi_{\ell,\ell'}f(z) = \int_{S^{2n+1}} f(w) Z_{\ell,\ell'}(z,w)\,
d\sigma(w) \, .
$$

\medskip

\subsection{\it{Classical and non-isotropic Sobolev spaces.}} 
Recall the decomposition \eqref{dec-L2} of $L^2(S^{2n+1})$.
Each subspace 
${\mathcal{H}}^{\ell,\ell'}$ is an eigenspace both for Laplace--Beltrami operator $\Delta$ with 
eigenvalue $\mu_{\ell,\ell'}:=(\ell+\ell')(\ell+\ell'+2n)$,
and for $\mathcal L$ with eigenvalue
$\lambda_{\ell,\ell'}:=2\ell\ell'+n (\ell+\ell')$.
For these and other properties of $\mathcal L$ we refer the reader to
\cite{Geller} and \cite{RicciUnterberger}.
\medskip

The non-isotropic
Sobolev spaces on the complex sphere can be defined in terms of
suitable powers of $I+\mathcal L$, or,
equivalently, 
in   terms of suitable 
  powers of the conformal sublaplacian ${\mathcal{D}}:=\mathcal L+\frac{n^2}{2}$; see,
  for instance, 
\cite{Folland}.
More precisely, for $1\le p\le\infty$
we set
\begin{equation}\label{defSobolev}
W^{r,p}(S^{2n+1}):=\big\{
f\in L^p(S^{2n+1})\,:
(I+\mathcal L)^{r/2}f\in L^p
\big\}\,.
\end{equation}
The operator 
$(I+\mathcal L)^{r/2}$ can be defined locally 
transferring the analogous operator from the Heisenberg group
via the Cayley transform,
see
\cite[\S~ 3]{Folland}, and also \cite{CaPe2}.
 
We will mostly deal with the case of  $L^2$-integrability, and we simply
 write $W^r$ for $W^{r,2}$.  For functions in $W^r$ we have the
 identity
$$
(I+\mathcal L)^{r/2} f
= \sum_{\ell,\ell'=0}^{+\infty} (1+\lambda_{\ell,\ell'})^{r/2}
\pi_{\ell,\ell'}f\, 
$$

Then, $W^r$ is a Hilbert space under the inner product 
$$
(f,g)_{W^r}:=\int_{S^{2n+1}}(I+\mathcal L)^{r/2}f\,
\overline{(I+\mathcal L)^{r/2}g}
\,.$$

For $s\ge0$, we
shall denote by
$H^s (S^{2n+1})$ the classical Sobolev space on $S^{2n+1}$, 
defined as  in \eqref{defSobolev},
with the operator $I+\mathcal L$ replaced by the  operator
$I+\Delta$. In particular, $H^s$ is endowed with the norm
$$
\|f\|_{H^s} = \Big(
\sum_{\ell,\ell'=0}^{\infty}
\big(1+
\mu_{{\ell,\ell'}}\big)^s
\,
\| \pi_{\ell,\ell'}f\|_{L^2}^2 \Big)^{1/2}\, .
$$
The following inclusions
follow
\begin{equation*}\label{inclusioniSob}
H^s\subseteq W^s\subseteq H^{s/2}\,.
\end{equation*}
For both isotropic and non-isotropic   Sobolev immersion theorems 
in a $\operatorname{CR}$ 
 setting  we refer to the seminal papers
\cite{Folland} and  \cite{FollandStein}, where results are proved 
in the framework
of Heisenberg groups.
Anyway, it is not difficult to check that the same inclusions 
hold on complex spheres. 

\subsection{\it{Mixed Sobolev spaces.}} 
We now introduce a family of Sobolev-type spaces 
that measure the
regularity of functions differently according to their spectral
localization. 

Fix 
a constant $M>1$
 and define
the proper cone $\mathcal V=\mathcal V_M$  in ${\mathbf{N}}^2$ 
\begin{equation}\label{def-cono}
\mathcal V:=\{(\ell, \ell'):\, 
 \ell/M < \ell'< M\ell\}
\end{equation}
and the pair of edges  $\mathcal E=\mathcal E_M$
\begin{equation}\label{def-spigolo}
\mathcal E:=\{(\ell, \ell'):\, \ell'\le \ell/M
\text{ or }
\ell'\ge M \ell
\}\,.
\end{equation}
We define the corresponding spectral
projections 
\begin{equation*}\label{spe-pro}
\pi_\mathcal V  = \sum_{\ell/M < \ell' < M\ell  }\pi_{\ell,\ell'} 
\quad\text{and}\quad
\pi_\mathcal E  = \sum_{ \ell'\le \ell/M\, \text{ or }\,  \ell'\ge M \ell } \pi_{\ell,\ell'} 
\, .
\end{equation*}
 We then  introduce the corresponding spaces
of spectrally localized functions 
\begin{equation*}
L^2_\mathcal V(S^{2n+1}) =\big\{ u\in L^2 (S^{2n+1} )  \,:\, 
u= \pi_\mathcal V u \big\}
\end{equation*}
and 
\begin{equation*}
L^2_\mathcal E(S^{2n+1}) =\big\{ u\in L^2 (S^{2n+1} )  \,:\, 
u=\pi_\mathcal E u 
\big\}\,.
\end{equation*}

We define the mixed Sobolev spaces $\mathcal X^{(r,s)}=\mathcal X^{(r,s)}_M(S^{2n+1})$
as
\begin{equation}\label{mix-Sob}
\mathcal X^{(r,s)}= \big\{ u\in L^2 (S^{2n+1}):\,  \pi_\mathcal V u\in W^r(S^{2n+1} )\ \text{and}\
\pi_\mathcal E u \in W^s(S^{2n+1} ) \big\}
\, ,
\end{equation}
with norm given by
$$
\|u\|_{\mathcal X^{(r,s)}} 
=\Big( \sum_{(\ell,\ell')\in\mathcal V} (1+\lambda_{\ell,\ell'})^r 
\| \pi_{\ell,\ell'} u\|_{L^2}^2 +
\sum_{(\ell,\ell')\in\mathcal E} (1+\lambda_{\ell,\ell'})^s 
\| \pi_{\ell,\ell'} u\|_{L^2}^2 
\Big)^{1/2} 
$$
Notice that the norm depends on $M$ although it will not be
explicitely indicated.

In general,
given a function space $\mathcal Y\subseteq L^2 (S^{2n+1})$,
we denote by $\mathcal Y_\mathcal V$ and $\mathcal Y_\mathcal E$ respectively, the subspaces
of $\mathcal Y$ of the functions that are spectrally localized
in $\mathcal V$ and $\mathcal E$, respectively.  Then, we have
$$
\mathcal X^{(r,s)} = W^r_\mathcal V \cap W^s_\mathcal E\, .
$$

For the mixed Sobolev spaces $\mathcal X^{(r,s)}$ 
we have 
the following elementary result that gives
embedding in the Lebesgue spaces and comparison with the classical non-isotropic
Sobolev spaces.

\begin{proposition}\label{immersioni}
Let $M>1$ be fixed.  Given $r,s\ge0$ 
the following properties hold true.
\begin{itemize}
\item[(1)] If $\min(r,s)>Q(\frac12-\frac1q)$ then for all $u\in
  {\mathcal{C}}^\infty(S^{2n+1})$ we have
$$
\|u\|_{L^q}\leq C \|u\|_{\mathcal X^{(r,s)}}\, .
$$
\item[(2)]  If 
$u\in {\mathcal{C}}^\infty_\mathcal V$, then 
\begin{equation*}
\|u\|_{\mathcal X^{(r,s)}}\approx \|u\|_{W^r}\approx \|u\|_{H^r}\, ,
\end{equation*}
and, for $\min(r,s)>(2n+1)(\frac12-\frac1q)$
\begin{equation*}
\|u\|_{L^q}  \le C\|u\|_{\mathcal X^{(r,s)}}
\, .
\end{equation*}
\item[(3)]
For all $u\in {\mathcal{C}}^\infty$ such that $\pi_{\ell,\ell'} u=0$ for
$\min(\ell,\ell')> M$ we have
\begin{equation*}
\|u\|_{\mathcal X^{(r,s)}} = \|u\|_{W^r}\approx \|u\|_{H^{r/2}}\,.
\end{equation*}
\end{itemize}
The constants involved 
in the above estimates depend on $M$.
\end{proposition}

\begin{proof}
(1) 
It suffices to recall the embedding theorems 
for the non-isotropic Sobolev spaces 
$W^r$.  
 Thm. 5.15 in \cite{Folland}
entails that, if $u\in
{\mathcal{C}}^\infty_\mathcal V$, then
$\|u\|_{L^q} \le C \|u\|_{W^{r,2}}$
if $r>Q(\frac12-\frac1q)$.
The result now follows easily.

Next, for
$u\in  {\mathcal{C}}^\infty_\mathcal V$ we have
\begin{align*} 
\|u\|_{(r,s)}^2
& \approx 
\sum_{\ell/M< \ell'< M\ell}
(1+\lambda_{\ell,\ell'})^r
\|\pi_{\ell,\ell'} u\|_2^2\approx
\sum_{\ell/M< \ell'< M\ell}
(1+\ell)^{2r}
\|\pi_{\ell,\ell'} u\|_2^2
\\
& 
\approx
\sum_{\ell/M< \ell'< M\ell}
(1+\mu_{\ell,\ell'})^{r}
\|\pi_{\ell,\ell'} u\|_2^2\, .
\end{align*} 
The second statement in
(2)  now follows
 from 
 the classical embedding 
theorem for the Sobolev space $H^\sigma(S^{2n+1})$.

Finally, (3) follows at once since, for $\min(\ell,\ell')\le M$, 
$(1+\lambda_{\ell,\ell'})\approx
(1+\mu_{\ell,\ell'})^{1/2}$.
\end{proof}


\section{The dispersive estimate}\label{dispersiveprima}\medskip

In this section we study   the  dispersive properties for
solutions of the Schr\"odinger equation 
that are spectrally localized.

The solutions of \eqref{schreq} do not satisfy in general a
dispersive estimate,
either 
globally (as constant solutions show for large $t$)
or locally in time.
A similar lack of the 
dispersive effect was noticed by
Burq, G\'erard and
Tzvetkov 
in the Riemannian case on compact manifolds as well (see \cite{Burq2}).

Indeed, if we could prove 
a dispersive estimate of the form
$$\| e^{it\mathcal L}\|_{(L^1 (S^{2n+1}), L^{\infty}(S^{2n+1}))}
\le \frac{C}{|t|^{q_0}}$$
for some $q_0>0$ and for some $t>0$,
then the $L^{\infty}$ norm
of eigenfunctions of the sublaplacian
should be controlled by
the $L^1$ norm and this is not true
in general
(see Theorem 3.1 in 
\cite{Casarino2}, where bounds are proved for the $L^1$ norm of zonal functions).

However, it is possible to prove a family of 
dispersive estimates on small time intervals related to the
frequencies of the data, that will suffice for the proof of the Strichartz
estimates (see \cite{Bahouri}
and
\cite{Tataru} for a first application of this idea). \medskip

In this paper, we are able to prove different dispersive estimates for data
$v$ that are spectrally localized according to a {\em double} decomposition
of the spectrum.

We introduce a spectral cut-off.  Let  $\varphi$ be a non-negative
smooth function with support contained in the
interval $[a,b]$, with $0<a<b<\infty$.
and $h\in(0,1]$, we consider the operator
$$
\varphi (h^2\mathcal L)\,:\, L^2 (S^{2n+1})\to L^2 (S^{2n+1})\, ,
$$
defined by the functional calculus for the sublaplacian.  

Next, we fix a smooth cut-off function $\psi$ with compact support in
$[1/M,M]$, where  $M>1$ is a (large) constant.

\begin{theorem}\label{dispersive}
Let $\varphi,\psi$  be  smooth cut-off functions 
be defined as above. 
Let $p, p'$ be
such that $\frac{1}{p}+\frac{1}{p'}=1$, $p\in [1,2]$. Then the
following estimates hold.
\begin{itemize}
\item[(i)]
Let  $s_n$ is as in \eqref{def-sn} and let $s>s_n$ 
 if $n>1$, 
or
$s\ge s_n$ if $n=1$.
Then there  exist  $c, C_s>0$ 
such that for all $v_0\in{\mathcal{C}}^{\infty}(S^{2n+1})$,
for all $h\in (0,1]$
\begin{equation}\label{stimadisp-cono}
\Big\|
 \sum_{\ell, \ell' \ge 0 } 
 e^{it{\lambda}_{\ell,\ell'} }\varphi (h^2{\lambda}_{\ell,\ell'})\psi(\ell'/\ell)\pi_{{\ell,\ell'}}(v_0) 
\Big\|_{L^{p'}(S^{2n+1})}
\le
\frac{C_s }{|t|^{Q/2 (\frac{1}{p}-\frac{1}{p'})}}\|v_0\|_{L^{p}(S^{2n+1})}
\end{equation} 
for all
 $t\in I_s :=[-c h^s,c h^s]$.
\item[(ii)]
 Then there 
exists  $C>0$ 
such that for all $v_0\in{\mathcal{C}}^{\infty}(S^{2n+1})$,
for all $h\in (0,1]$
\begin{equation}\label{stimadisp-spigoli}
\Big\|
 \sum_{\ell, \ell' \ge 0 } 
 e^{it{\lambda}_{\ell,\ell'}}\varphi (h^2 {\lambda}_{\ell,\ell'})\big( 1-\psi(\ell'/\ell)\big)\pi_{{\ell,\ell'}}(v_0)
\Big\|
 _{L^{p'}(S^{2n+1})}\le
\frac{C_s }{|t|^{Q/2 (\frac{1}{p}-\frac{1}{p'})}}\|v_0\|_{L^{p}(S^{2n+1})}
\end{equation} 
for all
 $t\in I_2 :=[-h^2,h^2]$.
\end{itemize}
\end{theorem}

\begin{remark}\label{timeinterval}
We point out that the index $s$ which determines the length of the
time interval $I_s$  in \eqref{stimadisp-cono} and \eqref{stimadisp-spigoli}, is subject to 
the following upper bound.
For, the kernel of the operator
$e^{it\mathcal L} \varphi (h^2 \mathcal L)$
is given by
$$
K_h (t, z,w)
=
\sum_ {\ell,\ell'=0}^{\infty}
e^{it{\lambda}_{\ell,\ell'}} \varphi (h^2 {\lambda}_{\ell,\ell'})
Z_{\ell,\ell'} (z,w)\, .
$$
Reasoning as in \cite{Burq2}, we have
\begin{align*}
\| e^{it \mathcal L}\varphi (h^2 \mathcal L)\|_{(L^1,L^\infty)}
&= \|K_h (t,\cdot,\cdot)\|_{L^{\infty}(S^{2n+1}\times S^{2n+1})}\ge
C\|K_h (t,\cdot,\cdot)\|_{L^{2}(S^{2n+1}\times S^{2n+1})}\\
&\ge C \Big( \sum_ {\ell,\ell'=0}^{\infty}
|\varphi (h^2 {\lambda}_{\ell,\ell'})|^2 d_{\ell,\ell'} \Big)^{1/2}
\ge  \frac{C}{h^{n-1}}
\Big( \sum_ {{\lambda}_{\ell,\ell'}\sim h^{-2}} (\ell+\ell') \Big)^{1/2}\\
&\ge \frac{C}{h^{n-1}}
\Big( \sum_ {\ell\sim h^{-2}} \ell \Big)^{1/2}
= \frac{C}{h^{n+1}}
\end{align*}
where in particular we used
\eqref{dimensione}.
Then estimates like
\eqref{stimadisp-cono}
or
\eqref{stimadisp-spigoli}
for $p=1$
imply
$|t|\le c h$ for some $c>0$, that is, $s\ge
1$.
\medskip
\end{remark}

\subsection*{Proof of Theorem \ref{dispersive}}
For all $t\in {\mathbf{R}}$ we have
$$
\big\|
 e^{it\mathcal L}\varphi (h^2 \mathcal L)v_0\big\|_{L^{2} }\le
 {C}{\|v_0\|_{L^2}}\,.
 $$
Thus, 
as a consequence of 
the Riesz-Thorin Theorem
 and
Young's inequality, 
it suffices to prove the following estimates
\begin{equation}\label{normaLinfinitogenerale-cono}
\Big\|
 \sum_{\ell, \ell'\ge 0} 
 e^{it{\lambda}_{\ell,\ell'}}\varphi (h^2 {\lambda}_{\ell,\ell'}) \psi(\ell'/\ell) Z_{\ell,\ell'}
\Big\|_{L^{\infty}(S^{2n+1}\times S^{2n+1})}\le
 \frac{C}{|t|^{Q/2}}
 \end{equation} 
for all
$|t|\le h^s$, where $s>s_n:=2[1-1/(n+1)]$ and
\begin{equation}\label{normaLinfinitogenerale-spigoli}
\Big\|
 \sum_{\ell, \ell'\ge 0} 
 e^{it{\lambda}_{\ell,\ell'}}\varphi (h^2 {\lambda}_{\ell,\ell'}) \big(
 1-\psi(\ell'/\ell)\big) Z_{\ell,\ell'}
\Big\|_{L^{\infty}(S^{2n+1}\times S^{2n+1})}\le
 \frac{C}{|t|^{Q/2}}
 \end{equation} 
for all
$|t|\le h^2$.

\medskip

In order to prove Theorem
\ref{dispersive} we break the proof of \eqref{normaLinfinitogenerale-cono}
and \eqref{normaLinfinitogenerale-spigoli}
into a few steps, that now we summarize.

{\bf{Step 1.}} We prove both estimates when (i) $h\ge \varepsilon_0$, where
  $\varepsilon_0>0$ is a fixed constant and (ii)
when $|t|\le h^2$.  This second case proves in fact that
  both \eqref{normaLinfinitogenerale-cono} and
  \eqref{normaLinfinitogenerale-spigoli} hold for these values of $t$ and
   in particular establishes \eqref{stimadisp-spigoli}
in Theorem  \ref{dispersive}.


{\bf{Step 2.}} Recalling 
\eqref{zonali} and \eqref{cos-eitheta}, we prove
\eqref{normaLinfinitogenerale-cono} 
 when $\langle z,w\rangle=e^{i\omega}\cos\theta$ varies in a
fixed compact set of the unit disk, that is, when $\theta\in
[\varepsilon_1,\pi/2]$, for some $\varepsilon_1>0$.

{\bf{Step 3.}} Next, we assume that $h^2\le |t|\le h^s$, and
  $0<h<\varepsilon_0$ and prepare to estimate
$$
\Big|  \sum_ {{\ell,\ell'=1}}^{+\infty} 
  e^{it{\lambda}_{\ell,\ell'}}\varphi (h^2{\lambda}_{\ell,\ell'} )\psi(\ell'/\ell)Z_{\ell,\ell'} (z,w)\Big|
$$
when $\langle z,w\rangle$ varies outside the compact set fixed in Step 2, that
is, when $\theta<\varepsilon_1$. 

First
we need to distinguish  between the diagonal case $\ell=\ell'$
and the sums  over  $\ell>\ell'$ and    $\ell<\ell'$.
To do this,  we  introduce
an even  cut-off function
  $\eta_0$,
  identically 1 for $|\xi|\le1/4$ and identically 0 for $|\xi|\ge1/2$,
  and the two cut-off functions
  $\eta_\pm(\xi) =\chi_{(0,+\infty)}\big[1-\eta_0(\pm\xi)\big]$, supported respectively on
  $\ell>\ell'$ and    $\ell<\ell'$.
Accordingly, we decompose the sum above as
$$
K^0 +K^+ +K^-\, 
$$
by writing
$1=\eta_0(\ell'-\ell)+\eta_+(\ell'-\ell)+\eta_-(\ell'-\ell)$.

The estimate for $K^0$ turns out to be trivial, while,
in order to estimate $K^\pm$,  we need another decomposition.
Clearly, it suffices to consider the case of $K^+$, which is supported
when $\ell'>\ell$.

We need to distinguish between the cases $\theta 
\le 1/\ell'$ (recall that $\ell'=\max\{\ell',\ell\}$) and 
$\theta>1/ \ell'$.
Thus we  introduce a
cut-off function $\chi_1$
supported in $[0,2]$, set $\chi_2=1-\chi_2$ and split $K^+$ as 
$K_1^+(\omega,\theta)+K_2^+(\omega,\theta)$, where, for $j=1,2$,
$$
K_j^+(\omega,\theta)
:= \sum_ {{\ell,\ell'=1}}^{+\infty} 
  e^{it{\lambda}_{\ell,\ell'}}\varphi (h^2{\lambda}_{\ell,\ell'} )\psi(\ell'/\ell) \eta_+(\ell'-\ell)
  \chi_j\big((\ell+\ell'+n)\theta\bigr) Z_{\ell,\ell'} (z,w)\,.
$$

{\bf{Step 4.}} 
Here we prove the estimate for $K_1^\pm (\omega,\theta)$.
 
{\bf{Step 5.}} Finally, we prove the estimate for $K_2^\pm (\omega,\theta)$
  and complete the proof of Theorem \ref{dispersive}.
\medskip

\noindent
{\bf{Step 1.}}
We begin proving that the estimate \eqref{normaLinfinitogenerale-cono} is
trivial in two cases: when $h\ge \varepsilon_0$ and 
in the low frequency case,
that is, when  $|t| \le h^2$.
In this case,
the oscillations of the exponential function are ineffective.

In \cite{CaPe}
the authors proved a restriction-type lemma for blocks of spectral
projections associated to the sublaplacian on $S^{2n+1}$. 
A key ingredient in the proof was the following estimate.
\begin{lemma}\label{sommeiperboli}
Let $1\le a<b$ be fixed.
Then there exists a constant $C>0$
depending only on $n$ such that 
\begin{equation}\label{sommaiperboleab}
\sum_{{{\lambda}_{\ell,\ell'}\in(a,b]}}
(\ell+\ell') \le
 C 
b 
\bigl(b-a +\log(b+1)\bigr)\, .
\end{equation}
\end{lemma}

 \begin{lemma}\label{dispersive_tpiccoli}
There exists a constant $C>0$ such that
the estimates \eqref{normaLinfinitogenerale-cono} 
and \eqref{normaLinfinitogenerale-spigoli} 
hold in the following cases:
\begin{itemize}
\item[(i)] when $t\in I_2$;\smallskip
\item[(ii)] when $h\ge \varepsilon_0$, for all $t\in I_s$;
\end{itemize}
where $I_2$ and $I_s$ are defined in Theorem \ref{dispersive}.
 \end{lemma}

\begin{proof}
For $z,w\in S^{2n+1}$,  
$|t|\le h^2$, using \eqref{2.7}
and
Lemma \ref{sommeiperboli} we have 
\begin{align*} 
\Big| 
\sum_{a/h^2< {\lambda}_{\ell,\ell'} <b/h^2}
 e^{it{\lambda}_{\ell,\ell'}}\varphi (h^2 {\lambda}_{\ell,\ell'})Z_{\ell,\ell'} (z,w)\Big| 
&\le 
 \sum_{a/h^2< {\lambda}_{\ell,\ell'} <b/h^2}
\frac{d_{\ell,\ell'}}{\omega_{2n+1}} 
\le   \sum_{a/h^2< {\lambda}_{\ell,\ell'} <b/h^2}
 {\lambda}_{\ell,\ell'}^{n-1}(\ell+\ell') \\
& \le \frac{C}{h^{2n-2}}
 \sum_{a/h^2< {\lambda}_{\ell,\ell'} <b/h^2}
(\ell+\ell')  \\
&
 \le  \frac{C}{|h|^{2n+2}}\,,
\end{align*}
for a suitable positive constant $C$.  

If $|t|\le h^2$, 
conclusion (i) follows at once.  If $h\ge \varepsilon_0$, (ii) also follows
at once, since 
$$h^{-(2n+2)}\le \varepsilon_0^{-2}h^{-2n}\le \varepsilon_0^{-2} |t|^{-Q/2}\, .$$
This proves the lemma.
\end{proof}

Observe that this lemma in particular proves the estimate
contained in \eqref{normaLinfinitogenerale-spigoli}-- in fact this is
the trivial part of the estimate, and it does provide no improvement
with respect to the Sobolev embedding theorem.
\medskip

\noindent 
{\bf{Step 2.}}
The dispersive estimate \eqref{stimadisp-cono} 
follows easily also when $\langle z,w\rangle$ varies in 
a  compact subset of 
the unit disk, as a consequence of the following result. First
we recall the following estimates for Jacobi polynomials,
see e.g. [BoCl, page 231],
\begin{equation}\label{inequalities}
\Big| P^{(\alpha,\beta)}_{\ell}(\cos\theta)\Big|
\le
 \begin{cases}\!
C
{\ell^{ \alpha}}
&\,\text{if
  $0\le \theta\le {\frac{\pi}{2}} $, }\cr
C \ell^{-1/2} \theta^{-\alpha-\frac12}
&\,\text{if ${0<
\theta\le {\frac{\pi}{2}}
}$,}\cr
C \ell^{-1/2}
|\pi-\theta|^{-\beta-\frac12}
&\,\text{if $
{\frac{\pi}{2}}\le\theta<\pi$,}\cr
C
{\ell^{ \beta}}
&\,\text{if
  ${\frac{\pi}{2}}\le\theta\le\pi$. }\cr
\end{cases}
\end{equation}

 \begin{lemma}\label{dispersive_compatti}
Let $0<\varepsilon_1<\frac{\pi}{2}$ be fixed and set 
$$
\mathcal K_{\varepsilon_1}:= 
\big\{ (z,w)\in S^{2n+1}\times S^{2n+1}:\, \langle
z,w\rangle=e^{i\omega}\cos\theta,\ 
\omega\in [0,2\pi]\,,\, \theta\in [\varepsilon_1,
\pi/2]\big\} \, .
$$ 
Then, there exists $C>0$ such that
\begin{equation}\label{normaLinfinitogeneraleCeo}
\sup_{ (z,w)\in {\mathcal K}_{\varepsilon_1}} \Big|
 \sum_{\ell, \ell'\ge 0}   e^{it{\lambda}_{\ell,\ell'}}\varphi (h^2 {\lambda}_{\ell,\ell'})Z_{\ell,\ell'}(z,w)
\Big| 
\le  \frac{C}{|t|^{Q/2}}\, .
 \end{equation} 
 \end{lemma}
 \begin{proof} 
The proof is simple since again in this case we do not need to
consider the oscillations of the kernel.
By symmetry in the parameters $\ell$ and $\ell'$, in
\eqref{normaLinfinitogeneraleCeo} it suffices 
to consider the case $\ell'\ge\ell$. 

Assume first that $\theta\in [\varepsilon_1, \frac{\pi}{4}]$.
In this case, for $ (z,w)\in {\mathcal K}_{\varepsilon_1}$, we have,
using
the first inequality in \eqref{inequalities} and \eqref{dimensione}
 \begin{align*} 
& \Big|
\sum_{\ell, \ell'\ge1}
 e^{it{\lambda}_{\ell,\ell'}}\varphi (h^2 {\lambda}_{\ell,\ell'})Z_{\ell,\ell'} (z,w)\Big| \\
&\quad \le C 
\sum_{\ell, \ell'\ge1}
\frac{d_{\ell,\ell'}}{\ell^{n-1}}(\cos\theta)^{\ell'-\ell}
\Big|P_{\ell}^{(n-1,{\ell'-\ell})}(\cos 2\theta) \Big|\\
& \quad \le
\frac{C}{h^{2n-2}} 
\sum_{a/h^2 \le {\lambda}_{\ell,\ell'}\le b/h^2}
(\ell+\ell') (\cos\theta)^{{\ell'-\ell}}\\
&\quad \le
\frac{C}{h^{2n-2}} 
\bigg(
\sum_{\ell=0}^{\lfloor c/h\rfloor}
\sum_{{\ell'-\ell}=0}^{\lfloor c/h\rfloor}
(\ell'-\ell)(\cos\theta)^{{\ell'-\ell}}
+
\sum_{\ell=0}^{\lfloor c/h\rfloor} 2\ell\,\sum_{{\ell'-\ell}=0}^{\lfloor c/h\rfloor}
(\cos\theta)^{{\ell'-\ell}}\bigg)
\le \frac{C}{h^{2n}}  \,,
 \end{align*}
since $\theta\in [\varepsilon_1, \pi/4]$.
 
Next, 
when $\theta\in[\pi/4, \pi/2) $, 
we observe that the sum vanishes when $\theta=\pi/2$ and then we
split it into two parts.  Recalling that we are assuming
$\ell\le\ell'$ we set
\begin{align*}
E_1 &= \big\{ (\ell,\ell')\in {\mathbf{N}}^2: \ a/h^2 < {\lambda}_{\ell,\ell'}<  b/h^2, \
\ell> 1/|\pi-2\theta|\big\}\, , \smallskip \\
E_2 &= \big\{ (\ell,\ell')\in {\mathbf{N}}^2: \ a/h^2 < {\lambda}_{\ell,\ell'}< b/h^2, \
\ell\le 1/|\pi-2\theta|\big\}\, .
\end{align*}
Then we have
\begin{align*} 
& \Big|
\sum_ {\ell'\ge\ell\ge1}
 e^{it{\lambda}_{\ell,\ell'}}\varphi (h^2 {\lambda}_{\ell,\ell'})Z_{\ell,\ell'} ((z,w))\Big| \\
& \quad \le 
\sum_{(\ell,\ell')\in E_1} 
\varphi (h^2 {\lambda}_{\ell,\ell'})\big| Z_{\ell,\ell'} ((z,w))\big| +
\sum_{(\ell,\ell')\in E_2} 
\varphi (h^2 {\lambda}_{\ell,\ell'})\big| Z_{\ell,\ell'} ((z,w))\big|\\
& \quad =:
S_1+S_2\, .
\end{align*}

Using the last inequality in \eqref{inequalities} we have
\begin{align*}
S_1&\le
C 
\sum_{(\ell,\ell')\in E_1}
\frac{d_{\ell,\ell'}
}{\ell^{n-1}}(\cos\theta)^{{\ell'-\ell}}
\Big|P_{\ell}^{(n-1,{\ell'-\ell})}(\cos 2\theta) \Big|
\\
&\le
C
\sum_{(\ell,\ell')\in E_1}  
 \frac{d_{\ell,\ell'}
}{\ell^{n-1}}
(  \pi/2-\theta)^{{\ell'-\ell}}
{\ell}^{{\ell'-\ell}} \\
&\le
\frac{C}{h^{2n-2}} 
\sum_{\ell=0}^{\lfloor c/h\rfloor}\sum_{{\ell'-\ell}=0}^{\lfloor c/h\rfloor}
\Big(
 \frac{1}{\ell^{n-1}}{{{(\ell'-\ell)}}}\,2^{-{(\ell'-\ell)}}
 +
 \frac{1}{\ell^{n-2}}\,2^{-{(\ell'-\ell)}}\Big)\\
& \le \frac{C}{h^{2n}}  \, .
 \end{align*}

An analogous bound may be proved for $S_2$ using
the third 
inequality in \eqref{inequalities}, we finally obtain
\eqref{normaLinfinitogeneraleCeo}.
 \end{proof} 
\medskip

\section{The main estimate}\label{Steps3-4}
\medskip

\noindent {\bf{Step 3.}}
Now we turn to the estimate
\eqref{stimadisp-cono}.
As a consequence of what has been proved in
Steps 1-2, from now on we may assume
that $h$ is
sufficiently small, and precisely that
$0<h<\varepsilon_0$,
and that  $t\in A$ where
\begin{equation}\label{A-def}
A:
=
\big\{ t:\,  h^2\le |t|\le  c h^s\big\}\, ,
\end{equation} 

Recall also that, because of the presence of the cut-off function
$\psi$ in \eqref{stimadisp-cono}, we may consider the parameters
$\ell,\ell'$ to be such that
$1/M<\ell'/\ell\le M$, where $M>1$ is a fixed (large) constant.

Starting from \eqref{normaLinfinitogenerale-cono} we now wish to show
that for every $\kappa>0$ there exists $C>0$ such that
\begin{equation}\label{normasemplificata}
\sup_{(z,w)\in\Omega} \Big|
 \sum_ {{\ell,\ell'=1}}^{+\infty} 
  e^{it{\lambda}_{\ell,\ell'}}\varphi (h^2{\lambda}_{\ell,\ell'} )\psi(\ell'/\ell)Z_{\ell,\ell'} (z,w)\Big|\le C\frac{1}{h^{2n+\kappa}}\, ,
\end{equation} 
for all $t\in A$, where 
$$
\Omega= 
\big\{ (z,w)\in S^{2n+1}\times S^{2n+1}:\, \langle z,w\rangle e^{i\omega}\cos\theta\, ,\ \text{where\ }  0\le \theta\le \varepsilon_1\,,\
\omega\in [0,2\pi)\big\} \, .
$$

We shall need to differentiate the proof between the cases $n=1$ and $n>1$
only at the end of Step 5, so that we will not distinguish 
between  different values of $n$ until 
Proposition \ref{estimate-Sigma-pm}.

We notice that we may assume $t>0$, since passing to the complex
conjugate in \eqref{normasemplificata} would change $Z_{\ell,\ell'}$ into
$Z_{\ell',\ell}$. \medskip

It turns out to be  convenient to  further simplify the problem, by separating
the cases $\ell=\ell'$,
$\ell<\ell'$ and $\ell>\ell'$.

We can do this by introducing yet
another cut-off function. 
 We let $\eta_0$ be an even cut-off function, 
identically 1 for $|\xi|\le1/4$ and identically 0 for $|\xi|\ge1/2$.
Then we write $1=\eta_0(\xi) +\eta_-(\xi)+\eta_+(\xi)$, where
$\eta_\pm(\xi)=\big(1-\eta_0(\xi)\big)\chi_{[0,+\infty)}(\pm\xi)$.
Accordingly, we decompose the sum in \eqref{normasemplificata}
 as
\begin{equation}\label{third-spe-dec}
K^0 +K^+ +K^-\, .
\end{equation}

It is easy to estimate $K^0$, since it  coincides with  the  sum in
\eqref{normasemplificata} restricted to the diagonal terms
$\ell=\ell'$ and  in this case the 
sum reduces to a summation in one variable.
\begin{lemma}\label{diag-case}
For all $t\in A$ we have
\begin{equation}\label{norma-diag}
\sup_{(z,w)\in\Omega} \Big|
 \sum_ {{\ell, \ell'=1}}^{+\infty} 
  e^{it\lambda_{\ell,\ell'}}\varphi (h^2\lambda_{\ell,\ell'} )\eta_0
  (\ell-\ell')
Z_{\ell,\ell'} (z,w)\Big|\le C\frac{1}{h^{2n}}\, .
\end{equation} 
\end{lemma}
\begin{proof}
We easily check that 
\begin{align*}
\sup_{(z,w)\in\Omega} \Big|
 \sum_ {{\ell, \ell'=1}}^{+\infty} 
  e^{it\lambda_{\ell,\ell'}}\varphi (h^2\lambda_{\ell,\ell'} )
\eta_0 (\ell-\ell')Z_{\ell,\ell'} (z,w)\Big|
&=\sup_{(z,w)\in\Omega} \Big|
 \sum_ {{\ell=1}}^{+\infty} 
  e^{it\lambda_{\ell,\ell}}\varphi (h^2\lambda_{\ell,\ell} )Z_{\ell,\ell} (z,w)\Big|\\
 &\le C  \Big|
 \sum_ {{\ell=1}}^{+\infty} 
\varphi (h^2\lambda_{\ell,\ell} )(\lambda_{\ell,\ell})^{n-1}\ell \Big|\le C\frac{1}{h^{2n}}\, \\
\end{align*} 
for all $t\in A$.
\end{proof}

We are left with the estimate of 
$K^{\pm}(\omega,\theta)$, where 
\begin{align}
K^{\pm}(\omega,\theta) :=  \sum_ {\ell,\ell'=0 }^{\infty}
e^{it{\lambda}_{\ell,\ell'}} \varphi (h^2 {\lambda}_{\ell,\ell'})
\eta_\pm(\ell'-\ell)
\psi(\ell'/\ell)Z_{\ell,\ell'} (z,w)
 \, .
\end{align}

Our proof of
 \eqref{normasemplificata}, with the inner sum replaced by $K^\pm$,
hinges on 
Lemma \ref{stimeFitohui} below
and
on Lemma \ref{Poisson}, which will be proved in Step 4.

First,
we  need
a representation of the Jacobi polynomials $P^{(\alpha,\beta)}_d$
showing explicitly the
dependence on the parameters.  To this end, we use a very precise representation of 
$P^{(\alpha,\beta)}_d$
 due to A. Fitohui and M. M. Hamza
\cite{FH}.
We denote by $J_\nu$ the Bessel function of order $\nu$.

\begin{lemma}\label{stimeFitohui}
Let $\alpha>-\frac{1}{2}$, $\beta>-1$, $d$ a positive integer. Then
\begin{multline}
\big( \sin\theta\big)^{\alpha+1/2}
\big( \cos \theta\big)^{\beta+1/2} 
P^{(\alpha,\beta)}_{d}(\cos 2\theta) \\
=\frac{\Gamma (d+\alpha+1)}{d!} \theta^{1/2}
\Bigg(
\sum_{p=0}^{m} \theta^p Q_{2p}(\beta,\theta)
\frac{J_{\alpha+p}(N \theta)}{N^{\alpha+p}}
+ \theta^{m+1}\mathcal R_{m,N}(\theta)\Bigg) \, , \label{Jacobi-expansion}
\end{multline}
where 
$N:= 2d+ \alpha+\beta+1$,
 the functions $Q_{2p}(\beta,\theta)$ are  polynomials of degree $2p$ in
$\beta$ and
analytic in $\theta\in[0,\pi/2)$, and
 \begin{equation}\label{restoFitohui}
\mathcal R_{m,N}=\mathcal O \big(
N^{-(\alpha+m+3/2)}\big)\,,
\end{equation}
as $N\to+\infty$, uniformly in $\theta\in[0, \frac\pi2-\tilde\varepsilon]$, $\tilde\varepsilon>0$ being arbitrary.
 \end{lemma}
 
\begin{proof}
This is just a restatement of Theorem 4 in  \cite{FH}; notice however
that the parentheses are missing on the right hand side of 
(6.6) in \cite{FH}. 
 By formula
(6.6) in \cite{FH}, setting $x=2\theta$ and $B_p(x)=Q_{2p}(\beta,\theta)$ (and
calling $N$ what is $2N$ in \cite{FH}) we immediately obtain
\eqref{Jacobi-expansion}. \medskip

Next, the functions $B_p (x)$
are recursively defined by
  \begin{equation}\label{recursive}
B_0(x)=1\, ,\quad
\big( x^{p+1}B_{p+1}(x)\big)'
=
-\frac12 x^p
\Big(
B_p^{''}(x)+\frac{1-2\alpha}{x}
B^{'}_p (x)+\chi (x)B_p (x)\Big)\, ,
\end{equation}
where 
$$
\chi (x)
=
\bigg(\frac{1}{4}-\alpha^2\bigg)
\bigg( \frac{1}{4\sin^2 (x/2)}-\frac{1}{x^2}\bigg)
+\bigg( \frac14-\beta^2\bigg) \frac{1}{4\cos^2 (x/2)}\, ,
$$
and it turns out that the $B_p(x)$ are analytic for $x\in[0,\pi)$ (see
Section 6.3 in \cite{FH}).

From the recursive relation \eqref{recursive} it is easy to see that
the functions $B_p$ are polynomials of degree $2p$ in the index $\beta$.
(We point out that, for our purposes,
the dependence on $\alpha$ is not relevant, since $\alpha=n-1$ 
and it is not related to the  the indexes $\ell$, $\ell'$, while $\beta=|\ell'-\ell|$).
In fact, the recursive relation can be restated by saying that
$$-2B_{p+1}=H_{p+1} \big(
(L+\chi)
B_p\big)\,,$$
where
$H_{p+1}$ is the integral operator (3.3) in \cite{FH},
independent of $\alpha$ and $\beta$, and
$L$ is the differential operator given by
$L u=u''+\frac{1-2\alpha}{x}u'$.
Now, the statement about the dependence on $\beta$ follows easily by
induction. 
Hence,
$Q_{2p}(\beta,\theta)$ is 
a polynomial of degree $2p$ in $\beta$ and  analytic in 
$\theta\in 
[0, \frac{\pi}{2})$.

The statement about the remainder term $\mathcal R_{m,N}(\theta)$
is explicit in Theorem 4 (see formula (6.6) again) in \cite{FH}.
\end{proof}

We are going to apply Lemma
\ref{stimeFitohui}, so that we observe that in our case $\alpha=n-1$, $d=\min\{\ell,\ell'\}$ and
$\beta=|\ell'-\ell|$; hence
$N=\ell+\ell'+n$.  

In what follows we denote
by $g_j,\tilde g_{j'}$ polynomials
of degree $j,j'$ resp., in the indicated variables, that again may
have different expression from one line to the next.  Then,  we write 
\begin{align}
\frac{\ell+\ell'+n}{\ell'}
\begin{pmatrix} \ell'+n-1\\
	 \ell'-1\end{pmatrix}
	 =
& ({\ell+\ell'+n})
\frac{(\ell'+n-1)!}{\ell' ! n!}
=
N
g_{n-1} (\ell')
\, ,\label{sommafattoriali}\ 
	 \end{align} 
 and
\begin{equation}\label{sf2}
\frac{\Gamma (\ell+n)}{\ell!} =\tilde g_{n-1}(\ell)\, .
\end{equation}

Then, using 
\eqref{zonali}, \eqref{sommafattoriali} and \eqref{sf2}, writing 
$\langle z,w\rangle=e^{i\omega} \cos\theta$, in the case $\ell'\ge\ell$ we have
 \begin{align}\label{stimabordov}
Z_{\ell,\ell'} (z,w ) 
& = \frac{n}{\omega_{2n+1}}
 N g_{n-1}(\ell')
\, e^{i\omega (\ell'-\ell)} 
\big( \sin\theta \big)^{-n+1/2} \big( \cos \theta\big)^{-1/2} \tilde g_{n-1}(\ell) 
\theta^{1/2}\notag\\
& \qquad \qquad \qquad \times \Bigg(
\sum_{p=0}^{m}\theta^p
Q_{2p} (\beta,\theta)
\frac{J_{n-1+p}(N \theta)}{N^{n-1+p}}
+\theta^{m+1}\mathcal R_{m,N}(\theta)\Bigg) \notag\\
&= b(\theta)e^{i\omega (\ell'-\ell)} 
  N g_{n-1}(\ell')
 \tilde g_{n-1}(\ell) \notag \\
&\qquad\qquad\qquad \times \Bigg(
\sum_{p=0}^{m}\theta^{2p}
Q_{2p} (\beta,\theta)
\frac{J_{n-1+p}(N \theta)}{(N\theta)^{n-1+p}}
+\theta^{m-n+2}\mathcal R_{m,N}(\theta)\Bigg) \, ,
\end{align}
where $b$ denotes an entire function of $\theta$.

If $\ell\ge\ell'$ we simply switch the roles between $\ell$ and
$\ell'$ in the formula above.

Then, in order to estimate \eqref{normasemplificata}  it suffices  to bound 
the modulus of 
\begin{align*}
 e^{i[t{\lambda}_{\ell,\ell'}+\omega (\ell'-\ell)]} &\varphi (h^2 {\lambda}_{\ell,\ell'}) \psi(\ell'/\ell)\eta_\pm (\ell-\ell')
 N g_{n-1}(\ell') \tilde g_{n-1}(\ell) \\ 
&\times \Bigg(
\sum_{p=0}^{m}\theta^{2p}
Q_{2p} (\beta,\theta)
\frac{J_{n-1+p}(N \theta)}{(N\theta)^{n-1+p}}
+\theta^{m-n+2}\mathcal R_{m,N}(\theta)\Bigg) \, \medskip
.
\end{align*}

We need to distinguish the cases when $N\theta$ remains bounded and
when it is bounded from below.  Then, let $\chi_1$ be a smooth cut-off
function with compact support such that $0\le\chi_1\le 1$,
$\chi_1(x) =1$ for $0\le x\le
1$ and $\chi_1(x)=0$ for $x\ge2$, and set $\chi_2= 1-\chi_1$.
Therefore, for $j=1,2$  we write
\begin{multline}
K_j^\pm(\omega,\theta):= \sum_{a/h^2< {\lambda}_{\ell,\ell'} <b/h^2}
 e^{i[t{\lambda}_{\ell,\ell'}+\omega(\ell'-\ell)]} \varphi (h^2 {\lambda}_{\ell,\ell'}) \psi(\ell'/\ell)\eta_\pm (\ell-\ell')
 N g_{n-1}(\ell') \tilde g_{n-1}(\ell) \chi_j(N\theta) \\
\times \Bigg(
\sum_{p=0}^{m}\theta^{2p}
Q_{2p} (\beta,\theta)
\frac{J_{n-1+p}(N \theta)}{(N\theta)^{n-1+p}}
+\theta^{m-n+2}\mathcal R_{m,N}(\theta)\Bigg) \, \medskip
. \label{Lj}
\end{multline}

\begin{remark}\label{change-of-parameters}{\rm
We observe that the eigenvalue 
${\lambda}_{\ell,\ell'} =2\big[ (\ell+n/2)(\ell'+n/2)- n^2/4\big]$ and we set
\begin{equation}\label{new-par}
k=\ell+\frac{n}2\, ,\qquad k'=\ell'+\frac{n}2\, .
\end{equation}
Notice that we may consider the quantities $g_{n-1},\tilde g_{n-1}$ as
functions of $k,k'$ resp., and write $N=k+k'$.  
We adopt the convention that, if $n$ is odd, then the symbol
$  \sum_ {k,k'\ge 1 }$
shall denote the sum
over a suitable subset of ${\mathbf{N}}$ shifted by $1/2$.

Moreover, 
the condition $h^2{\lambda}_{\ell,\ell'}\in\operatorname{supp}\varphi$ becomes
\begin{equation*}\label{new-supp-cond}
\frac{a_h}{h^2}\le kk'\le \frac{b_h}{h^2}\, ,
\end{equation*}
where we set
$$
a_h=\frac{a}{2} +\frac{h^2n^2}{4}\, , \quad\text{and}\quad
b_h=\frac{b}{2} +\frac{h^2n^2}{4}\, .
$$
For simplicity of notation we take $0<a'\le a_h$ and $b'\ge b_h$ for all
$h\le 1$.  We also set $c'=\sqrt{b'}$.  

The cut-off function $\varphi(h^2{\lambda}_{\ell,\ell'})$ can be written as
$$
\varphi(h^2 {\lambda}_{\ell,\ell'}) = \varphi \big(h^2 (2kk' -n^2/2)\big) =:
\varphi_h(h^2kk')\, .
$$
We remark that Lemma \ref{Poisson} holds true if the cut-off function
$\varphi$ is replaced by a family of functions $\varphi_\varepsilon$
converging in the 
Schwartz norms to $\varphi$ as $\varepsilon\to0$.  Since the dependence on
$h$ of $\varphi_h$  is ineffective,
with an abuse of notation,
we write again $\varphi$ to denote the functions $\varphi_h$.\medskip

The cut-off function $\psi(\ell/\ell') $, supported when $1/M \le
\ell/\ell'\le M$ is changed into 
$$\psi\big( (k-n/2)/(k'-n/2)\big)=:\tilde\psi(k,k')\, .  
$$
Observed that $\tilde\psi$ is a cut-off function having support
contained in the set $\{1\le k/k'\le M\}$.  Finally, notice that the
support condition of $\varphi$ implies that
\begin{equation*}\label{k-le-k'}
n/2\le k,k' \le c'/h\, .
\end{equation*}

With the change of parameters \eqref{new-par}, the quantity
$\beta$ remains unchanged, and the phase function
$t[{\lambda}_{\ell,\ell'}+\omega\beta]$ becomes $2t[kk'+\omega\beta-n^2/4]$ and we
may 
absorb the factor 2 in the parameter  $t$.
}
\end{remark}

\noindent {\bf{Step 4.}}
We now wish to estimate
 the
modulus of $K_1^\pm (\omega,\theta)$, as defined in \eqref{Lj}. 

We will consider the case of $K_1^+$, the other one being completely
analogous. 

It suffices to
estimate the modulus of
\begin{align} 
& \sum_{p=0}^{m}
\Bigg(  \sum_ {k,k'\ge n/2 }
 e^{i[tkk' +\omega (k'-k)]} \varphi (h^2 kk')
 \tilde\psi(k,k')\chi_1(N\theta) \eta_+ (k'-k)
N g_{n-1} (k')  \tilde g_{n-1}(k) \notag
\\
&\qquad\qquad\qquad\qquad\qquad\qquad\qquad\qquad\qquad\qquad\qquad\qquad\qquad\times
\theta^{2p}Q_{2p} (\beta,\theta)
\frac{J_{n-1+p}(N \theta)}{(N\theta)^{n-1+p}}\Bigg) \notag\\
& \quad + 
\sum_ {k,k'\ge n/2 }
 e^{i[tkk' +\omega (k'-k)]} \varphi (h^2
 kk')\tilde\psi(k,k')\eta_+ (k'-k) \chi_1(N\theta)
N g_{n-1} (k')  \tilde g_{n-1}(k) \theta^{m-n+2}\mathcal R_{m,N}(\theta) \notag\\
&=: \sum_{p=0}^{m} K_p^{1,+}(\omega,\theta)  +K_{\mathcal R}^{1,+}(\omega,\theta) 
\, .\label{sommaspezzata}
\end{align}

Since $ \mathcal R_{m,N}(\theta)=\mathcal O \big( \frac{1}{N^{n+m+1/2}}\big)$
uniformly for $\theta\in[0, \frac\pi2-\varepsilon_1]$, for $\varepsilon_1>0 $, we 
can easily estimate the modulus of the error term $K_\mathcal R^{1,+}(\omega,\theta,N)$
by simply taking the modulus inside the sum.   Observing that
$x\mapsto x^{n-1}\varphi(x)$ is also a smooth function with compact support, and
choosing
$m= \max(n-2,0)$
 \begin{align}
 \big|K_\mathcal R^{1,+} (\omega,\theta)\big|
 &
 =
  \Big|
\sum_ {k,k'\ge n/2 }
 e^{i[tkk' +\omega (k'-k)]} \varphi (h^2 kk')
 \tilde\psi(k,k')\eta_+ (k'-k) \chi_1(N\theta) 
\notag
\\
&\qquad\qquad\qquad\qquad\qquad\times
N g_{n-1} (k') \tilde g_{n-1}(k) \theta^{m-n+2}\mathcal R_{m,N}(\theta) \Big| \notag\\
&
\le \frac{C}{h^{2n-2}}
\sum_ {k=n/2 }^{\lfloor c'/h \rfloor}
\sum_{k'=n/2 }^{\lfloor c'/h \rfloor}
\Big|
 \varphi (h^2 kk') (h^2 kk')^{n-1} 
 \frac{1}{N^{n+m-1/2}} \Big| \notag \\
& \le
\frac{C}{h^{2n-2}}\sum_ {k=n/2 }^{\lfloor c'/h \rfloor} 
\sum_{k'=n/2 }^{\lfloor c'/h \rfloor}
\frac{1}{(k+k')^{1/2}} \notag \\
& \le \frac{C}{h^{2n-1/2}}
\, ,
\end{align}
uniformly in $\omega$ and $\theta\in[0,\varepsilon_1]$.
Therefore, with $m= \max(n-2,0)$ we have
\begin{equation}
\label{est-SigmaR}
\big|K_\mathcal R^{1,+} (\omega,\theta)\big|
\le \frac{C}{h^{2n-1/2}}
\end{equation}
for all
$n\ge 1$.\medskip

We turn to the estimate of the main term in \eqref{sommaspezzata},
with $m=\max (n-2, 0)$. \medskip

In order to take advantage of the oscillations of the kernel we need
the following estimate for oscillating sums.
We denote by $\hat
f(\xi)$ the Fourier transform of an integrable function $f$ and
defined by $\hat f(\xi)=\int_{\mathbf R} f(x) e^{- 2\pi ix\xi} dx$.
\medskip

\begin{lemma}\label{Poisson}
Let $\varphi\in  {\mathcal{C}}_0^{\infty}({\mathbf{R}})$,
$\operatorname{supp}\varphi\subseteq[a,b]$, 
$0<a<b<\infty$, and
let $\sigma$ be a symbol in ${\mathcal{S}}^0$. 
Let $\mu\in{\mathbf{R}}$ and set $\operatorname{dist}(\mu,{\mathbf{Z}})\ge \delta$, for some $\delta>0$.
Then for every $L>1$ 
there exists a positive constant $C_L>0$, depending only on $\varphi$
and $\sigma$, 
such that
for $\delta,\varepsilon>0$ with
$0<\varepsilon\le \delta$,
we have that
\begin{equation}\label{O(1)}\Big|
\sum_{k\in{\mathbf{Z}}}
e^{2\pi i \mu k}
\varphi (\varepsilon k)\sigma (k)\Big|
\le C_L
\max \bigg\{
\frac{\varepsilon^{L-1}}{\delta^L},1
\bigg\}\, ,
\end{equation}
as $\varepsilon\to 0$.
\end{lemma}

\begin{proof}
Let $\psi$, $\tau$
be such that $\varphi=\widehat\psi$, $\sigma=\widehat\tau$.

We consider first the case $\sigma=1$.
By the classical Poisson summation formula
\begin{align}\label{stimapoisson-like-2}
\big| \sum_{k\in{\mathbf{Z}}} e^{2\pi i \mu k} \varphi (\varepsilon k)\big|
&=\big| \sum_{k\in{\mathbf{Z}}} e^{2\pi i\mu  k} \widehat{\psi_{\varepsilon}}(k)\big|
= \big|\sum_{k'\in{\mathbf{Z}}} {\psi_{\varepsilon}}(k'+\mu )\big| \notag\\
&\le \frac{C_L}{\varepsilon} \sum_{k'\in{\mathbf{Z}}} 
\frac{1}{\big(1+ \frac{|k'+\mu|}{\varepsilon}\big)^L}\, ,
\end{align}
where
${\psi_{\varepsilon}}(x)=\varepsilon^{-1}
\psi ({\varepsilon}^{-1}x)$.

We may assume 
$|\mu|\le \frac12$, so that $\operatorname{dist}(\mu,{\mathbf{Z}})=|\mu|\ge \delta$ and 
\begin{align*} 
\sum_{k'\in{\mathbf{Z}}}
\frac{1}{\big(1+ \frac{|k'+\mu|}{\varepsilon}\big)^L }
&\le
\frac{1}{\big(1+\frac{|\mu|}{\varepsilon} \big)^L} +2
\int_{-\infty}^{ +\infty}
\frac{1}{\big(1+\frac{|x+\mu|}{\varepsilon}\big)^L}\, dx\\
&\le
\frac{1}{\big(1+\frac{\delta}{\varepsilon}\big)^L} +2\varepsilon 
\int_{-\infty}^{ +\infty}
\frac{1}{\big(1+|y|\big)^L}\, dy \, .
  \end{align*}
Therefore,
\begin{align*}
\big| \sum_{k\in{\mathbf{Z}}} e^{2\pi i \mu k} \varphi (\varepsilon k)\big|
&\le \frac{C_L}{\varepsilon}
\Bigg( \frac{1}{\big(1+\frac{\delta}{\varepsilon} \big)^L}
+2\varepsilon \int_{ -\infty}^{ +\infty}
\frac{1}{\big(1+|y|\big)^L}\, dy\Bigg) \\
&\le
C_L \frac{\varepsilon^{L-1}}{\big(\varepsilon+\delta \big)^L} + C\\
&\le C_L\max
\bigg\{ \frac{\varepsilon^{L-1}}{ \delta^L} , 1\bigg\} 
\, ,
\end{align*}
proving \eqref{O(1)} in the case $\sigma=1$.

Next we suppose that 
$\sigma$ belongs to classical symbol class ${\mathcal{S}}^0$.
Notice that
\begin{equation*}
\varphi (\varepsilon k)\sigma (k)
=
\widehat{\psi_{\varepsilon}}
(k)\widehat{\tau}(k)
=
\big(
\psi_{\varepsilon}*\tau)^{\widehat{\ }}(k)
=
\big(\big(\psi*\tau_{1/\varepsilon}\big)_{\varepsilon}
\big) ^{\widehat{\ }}(k)\, .
\end{equation*}
Thus we may repeat the previous arguments as in
\eqref{stimapoisson-like-2} to obtain
\begin{align}
\big|
\sum_{k\in{\mathbf{Z}}}
e^{2\pi i \mu k}
\varphi (\varepsilon k)\sigma (k)\big|
&=\big|
\sum_{k\in{\mathbf{Z}}}
e^{2\pi i \mu k} 
\big(\big(\psi*\tau_{1/\varepsilon}\big)_{\varepsilon}
\big) ^{\widehat{\ }}(k)
 \big|\notag
\\
&=\big|
\sum_{k'\in{\mathbf{Z}}}
\big(\psi*\tau_{1/\varepsilon}\big)_{\varepsilon}
 (k'+\mu)\big|\notag\\
&\le
\frac{C_L}{\varepsilon} \sum_{k'\in{\mathbf{Z}}}
\frac{1}{\big( 1+\frac{|k'+\mu|}{\varepsilon}\big)^L}\,,\label{PoissonMacLaurin}
\end{align}
where $C_L$ does not depend on $\varepsilon$
as long as the Schwartz  norms of
$(\psi*\tau_{1/\varepsilon})$
are uniformly bounded in $\varepsilon$.
This happens if and only if
the Schwartz norms of
$\big(\psi*\tau_{1/\varepsilon} \big)^{\widehat{\ } }$
are uniformly bounded in $\varepsilon$, as $\varepsilon \to 0$.
Now
$$
\big(\psi*\tau_{1/\varepsilon}\big)^{\widehat{\ } }(x)
=  \widehat{\psi} (x)\widehat{\tau_{1/\varepsilon}} (x) 
= \varphi (x) \sigma (x/\varepsilon)\, .
$$
Since $\sigma\in{\mathcal{S}}^0$,  is straightforward to check that
\begin{align*}
\Big| D^j_x \big( \varphi\sigma(\cdot/\varepsilon)\big)(x) \Big|
& 
=
\Big| \sum_{j'=0}^j c_{j'} \varphi^{(j-j')}(x)\frac{1}{\varepsilon^{j'}}
\sigma^{(j')} (x/\varepsilon)\Big| \\
&  \le C
\sum_{j'=0}^{j} 
\big| \varphi^{(j-j')}(x)\big| \frac{1}{\varepsilon^{j'}} \frac{1}{(1+ |x|/\varepsilon)^{j'}}\\
& \le C
\sum_{j'=0}^{j} 
\big| \varphi^{(j-j')}(x)\big| \frac{1}{(\varepsilon^{j'}+a)^{j'}} \, ,
\end{align*}
since $\operatorname{supp}\varphi\subset[a,b]$ and $a>0$.
The statement now follows.
\end{proof}
\begin{remark}
It is worth noticing that, by choosing as a symbol
$\sigma$ a smooth
 cut-off function
with compact support, the 
estimate \eqref{O(1)}
may be proved also for truncated sums.
\medskip
\end{remark}

\begin{lemma}\label{symbol-Sigmap1}
For $N=k'+k$ and $\beta=|k'-k|$ set
$$
\sigma_1(k,k') 
= \eta_+(k'-k) \tilde\psi(k,k') \theta^{2p}Q_{2p} (\beta,\theta)
\frac{J_{n-1+p}(N \theta)}{(N\theta)^{n-1+p}} \chi_1(N\theta)\, .
$$
Then $\sigma_1$ is a symbol of order 0 in $k'$, depending on the
parameters $\theta$ and $k$, with norm uniformly bounded in such
parameters.
\end{lemma}

\begin{proof}
We wish to show that, considering $k'=\xi$ as a continuous parameter,
$\sigma_1$ is a smooth function of $\xi$ and, for each non-negative
integer $k$ there exists a positive constant $C=C_k$, independent of
$k$ and $\theta\in [0,\pi/2-\varepsilon_1]$, such that
$$
\big|\partial_\xi^k \sigma_1(\xi)\big|
\le C (1+|\xi|)^{-k}\, .
$$
Notice that since $k'\ge1$ we may assume that we have extended
$\sigma_1$ to be identically $0$ when $\xi\le1/2$.

Since the Bessel function $J_\nu$
of integral order $\nu$ is analytic and has a zero of order $\nu$ at
the origin, it is clear that $\sigma_1$ is smooth and bounded
uniformly in $\theta$.  Moreover, recall
from Lemma \ref{stimeFitohui}
that $Q_{2p}(\beta,\theta)$ is 
a polynomial of degree $2p$ in $\beta$ and  analytic in 
$\theta\in 
[0, \frac{\pi}{2})$.  Hence, since
$\chi_1(N\theta)=0$ for $N\theta\ge2$, we have that, on the support of
$\chi_1$,
$\theta\le 2/N\le
C/\xi$, implying  that 
$$
| \theta^{2p}Q_{2p}(\beta,\theta)| \le C\frac{|Q_{2p}(\beta,\theta)|}{N^{2p}}
$$
which
 is bounded,
as $\xi\to+\infty$, uniformly in $\theta$.

Next we consider the derivatives.  If the derivative falls on the
factor $\theta^{2p}Q_{2p}(\beta,\theta)$ we simply lower the degree of
the
polynomial of $\xi$ and then obtain the estimate
$$
\big| \partial_\xi \big[\theta^{2p}Q_{2p}(\beta,\theta)\big]\big| 
\le C\frac{1}{\xi}\, ,
$$
as $\xi\to+\infty$, uniformly in $\theta$, as we required.  

It the derivative falls on the factor $\frac{J_{n-1+p}(N
  \theta)}{(N\theta)^{n-1+p}}$, since the derivative of the Bessel
function of order $\nu$ satisfies the identity $J_\nu'(z)-\frac\nu{z} J_\nu(z)+J_{\nu-1}(z)$, again we obtain that
$$
\Big| \partial_\xi \Big[ \frac{J_{n-1+p}(N
  \theta)}{(N\theta)^{n-1+p}} \Big]\Big| 
\le C \big(\theta + \frac1N\big) \le C \frac{1}{\xi}\, ,
$$
as $\xi\to+\infty$, uniformly in $\theta$.

If the derivative falls on $\chi_1$, it produces an extra
factor $\theta$, which is less than $ C/\xi$.

Hence,
$$
\big|\partial_\xi \sigma_1(\xi)\big|
\le C (1+|\xi|)^{-1}\, ,
$$
as $|\xi|\to+\infty$, uniformly in $\theta$.

Finally, if the derivative falls on $\tilde\psi$ it produces a factor
of the order of $k/\xi^2$ which is less or equal to $C/|\xi|$, as $\xi\to+\infty$.

The argument can be repeated for all higher order derivatives, so the
lemma is proven.
\end{proof}

We wish to  apply Lemma \ref{Poisson}, and this leads us to analize
the phase function
$t kk' +\omega (k'-k)$.  Recall that $A$ has been
defined in \eqref{A-def} and that, as observed earlier, by passing to
the complex conjugate, we may assume $t>0$.
We then introduce the
 set of indeces in ${\mathbf{N}}^2$
\begin{equation*}\label{sV}
\mathcal V
=
\big\{ (k,k'):\, k,k'\ge n/2\, ,\
a'/h^2\le kk' \le b'/h^2\, ,\ 1/M\le k/k'\le
M \,,\, |k-k'|\ge1/2\big\} \, .
\end{equation*}
We set moreover 
\begin{equation}\label{sV-pm}
\mathcal V_+=\mathcal V\cap\{(k,k'):\, k'>k\}\qquad\text{and}\qquad 
\mathcal V_-=\mathcal V\cap\{(k,k'):\, k'<k\}\,.
\end{equation}
Finally, we introduce the  space of parameters $(t, \omega)$
\begin{equation*}
R:= [h^2,h^s]\times[0,2\pi) 
\subset A\times
[0,2\pi)\,.
\end{equation*}

\begin{lemma}\label{phase-fnc}
On  $\mathcal V_+$ we set 
$\mu(k):= (tk+\omega)/2\pi$ and
$\mu'(k')=(\omega-tk')/2\pi$.
Then, there exist  a constant $0<\gamma<1$ and 
two regions $R_I,R_{I\!I}$
 in the $(t,\omega)$-space,  such that
$R \subseteq\cup_{} R_I$ and 
for all $(t, \omega)\in R_i$,  $i\in  \{I,II\}$,
 one between the following two conditions
\begin{itemize}
\item[(I)]
$\operatorname{dist}(\mu(k),{\mathbf{Z}})\ge \gamma tk$ for all $(k,k')\in\mathcal V_+$;\smallskip
\item[(II)]
$\operatorname{dist}(\mu'(k'),{\mathbf{Z}})\ge\gamma tk'$ for all $(k,k')\in\mathcal V_+$
\end{itemize}
holds.

Analogous statement holds in the case of  $\mathcal V_-$.
\end{lemma}

\begin{proof}
We begin by observing that $(k,k')\in\mathcal V_+$ implies that
$1\le k< k'\le \lfloor c'/h\rfloor$,
so that, for $t\in [h^2,h^s]$ we have
\begin{equation}\label{stima-tk}
\frac{h^2}{2\pi}\le \frac{1}{2\pi} tk\le \frac{1}{2\pi} tk'\le c h^{s-1}\le c_1/(4\pi)\, ,
\end{equation}
if $h\le\varepsilon_0$ is sufficiently small and  for some positive, small enough  $c_1$.

Let $R_I=\{(t,\omega)\in [h^2,h^s] \times [0,2\pi):\, 0<c_1<\omega<2\pi-c_1\}$,
and $c_1$ as above. 
Then (I) holds for $(t,\omega)\in R_I$
since for $m$ integer
$$
 \big| \frac{ tk+\omega }{2\pi}-m\big|
\ge \big|\frac{\omega}{2\pi}-m\big| - \frac{c_1}{4\pi}\ge
\frac{c_1}{4\pi}\, .
$$

Replacing $\omega$ by $2\pi-\omega$ we may assume now that
$-c_1<\omega<c_1$.  In this case, 
notice that
$\operatorname{dist}(\mu,{\mathbf{Z}})=\frac{1}{2\pi}|tk+\omega|$ and
 $\operatorname{dist}(\mu',{\mathbf{Z}})=\frac{1}{2\pi}|tk'-\omega|$.
 
Then, if $\omega>0$, 
 we have that  $\frac{tk+\omega}{tk}\ge 1$, so that in this case (I) holds.

If $\omega<0$, 
 we then have that
 $\frac{tk'-\omega}{tk'}\ge 1$, so that in this case (II) holds.
\end{proof}

\begin{proposition}\label{estimate-Sigmap1}
There exists a constant $C>0$ such that
$$
\big| K_p^{1,+}(\omega,\theta)\big|
\le C \frac{1}{|t|^{n+1}}\, ,
$$
uniformly in $\omega$ and $\theta\in[0,\varepsilon_1]$, for all 
$|t|\in [h^2, ch^s]$. 
\end{proposition}

\begin{proof}
Recall the definition of $K_p^{1,\pm} (\omega,\theta)$
introduced in 
\eqref{sommaspezzata}.

Using Lemma \ref{symbol-Sigmap1} we have that
$$
K_p^{1,+}(\omega,\theta)
= 
\sum_ {k,k'\ge 1 }
e^{i[tkk' +\omega |k'-k|]} \varphi (h^2
kk')
N g_{n-1} (k')  \tilde g_{n-1}(k) \sigma_1(k')\\
\, .
$$

Notice that
$ g_{n-1} (k') \tilde g_{n-1}(k) 
 = k^{n-1}{k'}^{n-1} +$ $\text{lower order terms}$.
Thus, we  estimate the higher order terms, the other ones being estimated
in a similar way, giving rise to a better bound.  Hence, it suffices
to estimate
\begin{equation}\label{suffices}
\Big| 
 \sum_ {k,k'\ge n/2 }
 e^{i[tkk' +\omega (k'-k)]} \varphi (h^2 kk')
  Nk^{n-1}{k'}^{n-1} \sigma_1(k') \Big|\, ,
\end{equation}
and we distiguish two cases according that condition (I) or (II) in
Lemma \ref{phase-fnc} holds, respectively.

\medskip

\noindent
{\bf Case (I).}
We assume that 
$\operatorname{dist}(\mu(k),{\mathbf{Z}})\ge \gamma tk$ for all $(k,k')\in\mathcal V_+$ and
for some
$0<\gamma<1$,
so that
$|tk+\omega|\ge 2\pi  \gamma tk$ for all $k\ge1$.
Starting from \eqref{suffices}, we wish to estimate
\begin{align}
& \Big| 
 \sum_ {k=n/2 }^{\lfloor c'/h \rfloor}
\sum_{k'=  n/2}^{\lfloor c'/h \rfloor}
 e^{i[tkk' +\omega (k'-k)]} \varphi (h^2 kk') 
N k^{n-1}{k'}^{n-1}  \sigma_1(k,k') \Big| \notag\\
& \le \frac{C}{h^{2n}} 
 \sum_ {k=n/2 }^{\lfloor c'/h \rfloor} \frac1k 
\Big| 
\sum_{k'= n/2}^{{\lfloor c'/h \rfloor}}  
 e^{ik' [tk+\omega]} \varphi (h^2 kk')
  \big(h^2k{k'}\big) ^n
\sigma_1(k,k') \Big| \notag\\
& \qquad\qquad +
 \frac{C}{h^{2n-2}} \sum_ {k=n/2 }^{\lfloor c'/h \rfloor}
k \Big| 
\sum_{k'=  n/2}^{\lfloor c'/h\rfloor}  
 e^{ik'[tk+\omega]} \varphi (h^2 kk')
 \big(h^2k{k'}\big) ^{n-1} \sigma_1(k,k') \Big| \, . \label{above}
\end{align}
Notice that 
it suffices to bound the first sum on the right hand side
of \eqref{above} above.

We now apply Lemma \ref{Poisson} to the inner sum of the first term 
 on the right hand
side of \eqref{above} above, with 
$$
\mu =[tk+\omega]/2\pi\, , \qquad
\delta =\gamma t k \ge 
\gamma h^2k=\varepsilon\, ,
$$ 
and
cut-off function $x\mapsto x^n \varphi$ and 
with the aid of Lemma \ref{symbol-Sigmap1}.

Hence we obtain that for every $L>1$  the
right hand side of \eqref{above} is less or equal to a constant times
\begin{align*}
\frac{C_L}{h^{2n}}  \sum_ {k=n/2 }^{\lfloor c'/h \rfloor}
\frac1k 
\max \bigg( \frac{\varepsilon^{L-1}}{\delta^L} ,1 \bigg)
& \le  \frac{C}{h^{2n}}  \sum_ {k=n/2 }^{\lfloor c'/h \rfloor}
\frac1k 
\bigg( \frac{(h^2k)^{(L-1)}}{(t k)^L} +1 \bigg) \\
& \le  \frac{C}{h^{2n}}  \bigg( \frac{h^{2(L-1)}}{t^L}
\sum_ {k=n/2 }^{\lfloor c'/h \rfloor} \frac1{k^{2}} +\log(1/h)\bigg) \\
& \le C\bigg(\frac{1}{t^{n+1}}+\frac{1}{h^{2(n+\kappa)}}\bigg)
\, ,
\end{align*}
for any given $\kappa>0$,
if we choose $L=n+1$. 

Therefore, for $t$ such that $h^2\le t\le  h^s$
for
every $s>s_n$, we have
$$
\frac{1}{h^{2(n+\kappa)}}\le  \frac{C}{t^{2(n+\kappa)/s}} \le
\frac{C}{t^{n+1}}\, .
$$
Hence, for all $t\in A$ we obtain
\begin{equation}\label{end-case-1}
\big| K_p^{1,+}(\omega,\theta)\big|
\le C \frac{1}{|t|^{n+1}} \, .
\end{equation}
This proves the statement in Case (I).
\medskip

\noindent 
{\bf Case (II).}
We now assume 
that
$\operatorname{dist}(\mu'(k'),{\mathbf{Z}})\ge\gamma tk'$ for all $(k,k')\in\mathcal V_+$
and for some
$0<\gamma'<1$, so that
$|tk'-\omega|\ge 2\pi(h^2k')^{\gamma'}$, 
all $k'\in{\mathbf{N}}$.
In this case, again we start from \eqref{suffices}, apply Lemma
\ref{Poisson} to the inner sum with
$$
\mu =[tk'-\omega]/2\pi\, , \qquad
\delta =\gamma tk'  \ge 
\gamma h^2k'=\varepsilon\, ,
$$ 
and
cut-off function $x^n \varphi$ and 
with the aid of Lemma \ref{symbol-Sigmap1}.

 We have that
\begin{align}\label{sommascambiata}
& \Big| 
\sum_ {k,k'\ge n/2}
 e^{i[tkk' +\omega (k'-k)]} \varphi (h^2 kk')
Nk^{n-1} {k'}^{n-1}  \sigma_1(k,k') \Big|\notag \\
& \le \frac{C}{h^{2n}} 
 \sum_ {k'=n/2 }^{\lfloor c'/h \rfloor}  \frac1{k'} 
\Big| 
\sum_ {k=n/2 }^{\lfloor c'/h \rfloor}
 e^{ik [tk'-\omega]} \varphi (h^2 kk')
 \big(h^2k{k'}\big)^n  \sigma_1(k,k')\Big| \notag\\
& \qquad+
 \frac{C}{h^{2n-2}} \sum_ {k'=n/2 }^{\lfloor c'/h \rfloor}
k' \Big| 
\sum_ {k=n/2 }^{\lfloor c'/h \rfloor}   
e^{ik [tk'-\omega]}\varphi (h^2 kk')
 \big(h^2k{k'}\big) ^{n-1} \sigma_1(k,k')  \Big|\,.\end{align}
As in the previous case we may limit ourselves to consider the first sum 
in \eqref{sommascambiata}, the estimate for the latter one being analogous.
We have
\begin{align*}
& \frac{1}{h^{2n}} \sum_ {k'=1 }^{\lfloor c'/h \rfloor} \frac1{k'} 
\Big| 
\sum_ {k=1 }^{\lfloor c'/h \rfloor}
 e^{ik [tk'-\omega]} \varphi (h^2 kk')
 \big(h^2k{k'}\big)^n  \sigma_1(k,k') \Big| \notag\\
& \qquad\le \frac{C}{h^{2n}}    \sum_ {k'=1 }^{\lfloor c'/h \rfloor+1}
\frac1{k'}  
\bigg( \frac{(h^2k')^{(L-1)}}{(tk')^{L}} +1 \bigg) \\
& \qquad\le\frac{C}{h^{2n}}  \bigg( \frac{h^{2(L-1)}}{t^L}
 \sum_ {k'=1 }^{\lfloor c'/h \rfloor} \frac1{{k'}^2} +\log(1/h)\bigg) \\
& \qquad\le C\bigg(\frac{1}{t^{n+1}}+\frac{1}{h^{2(n+\kappa)}}\bigg)
\, ,
\end{align*}
 choosing $L=n+1$, for any $\kappa>0$.

Arguing as in \eqref{end-case-1} in Case (I), for 
$h^2\le |t|\le h^s$ we obtain 
\begin{equation}\label{end-case-2}
\big|K_p^{1,+}(\omega,\theta)\big|
 \le C \frac{1}{|t|^{n+1}}\, .
\end{equation}
The result now follows. \medskip
\end{proof}

\noindent {\bf{Step 5.}}
Finally, we wish to estimate
 the
modulus of $K_2^\pm(\omega,\theta)$, as defined in \eqref{Lj}.
Again, we consider only the case of $K_2^+$.
 
In this case, it turns out that it suffices to
take $m=0$.  
We introduce the same spectral decomposition as in 
\eqref{third-spe-dec}.
Thus, we are led to consider
the remainder term 
\begin{align*}
 \big|K_\mathcal R^{2,+} (\omega,\theta)\big|
 & = \Big|
\sum_ {k,k'\ge n/2 }
 e^{i[tkk' +\omega (k'-k)]} \varphi (h^2 kk')
 \tilde\psi(k,k')\eta_+ (k'-k) \chi_2(N\theta) \\
&\qquad\qquad\qquad\qquad \times 
N g_{n-1} (k') \tilde g_{n-1}(k) \theta^{m-n+2}\mathcal R_{m,N}(\theta) \Big| \\
&
\le \frac{C}{h^{2n-2}}
\sum_ {k=n/2 }^{\lfloor c'/h \rfloor}
\sum_{k'=n/2 }^{\lfloor c'/h \rfloor}
\Big|
 \varphi (h^2 kk') (h^2 kk')^{n-1} 
 \frac{\theta^{m-n+2}}{N^{n+m-1/2}} \Big|  \\
&
\le
\frac{C}{h^{2n-2}}\sum_ {k=n/2 }^{\lfloor c'/h \rfloor} 
\sum_{k'=n/2 }^{\lfloor c'/h \rfloor}
\frac{\theta^{m+3/2}}{(N\theta)^{n-1/2}N^m}
\, ,
\end{align*}
\medskip
so that, by choosing $m=0$, we obtain
\begin{equation}
\label{est-SigmaR-2}
\big|K_\mathcal R^{2,+} (\omega,\theta)\big|
\le \frac{C}{h^{2n}}
\end{equation}
for all
$n\ge 1$,
uniformly in $\omega$ and $\theta\in[0,\varepsilon_1]$.

Thus, for $j=2$
we are led to consider the main term in
\eqref{Lj} , that is, 
\begin{align}\label{decomp-ypsilon}
 \Upsilon^+(\omega,\theta)
&
:=\sum_ {k,k'\ge n/2 }
 e^{i[tkk' +\omega |k'-k|]} \varphi (h^2 kk')
 \tilde\psi(k,k')\eta_+(k'-k) \chi_2(N\theta) \notag\\
 &\qquad\qquad\qquad\qquad\qquad\qquad\qquad\times
N g_{n-1} (k')  \tilde g_{n-1}(k) \frac{J_{n-1}(N \theta)}{(N\theta)^{n-1}} \, .
\end{align}

Recall that in this case  we have that $N\theta\ge1$.
Then we use the asymptotic expansion of the 
Bessel function $J_\nu$
\begin{equation*}\label{Bessel-expansion}
J_\nu(x) 
=\frac{1}{x^{1/2}} \rho_1(x) e^{ix}  +
\frac{1}{x^{1/2}} \rho_2(x) e^{-ix} +\mathcal O (x^{-3/2})\, ,
\end{equation*}
for some bounded functions $\rho_j$, 
and write $\Upsilon^+(\omega,\theta)=\Upsilon_1+\Upsilon_2+\Upsilon_3$,
where, for $j=1,2$
\begin{multline}
\Upsilon_j (\omega,\theta)
=   
\sum_ {k,k'\ge n/2 }
e^{i[tkk' +\omega (k'-k)\pm\theta(k+k')]} \varphi
(h^2 kk') \tilde\psi(k,k') \eta_+(k'-k) \\
\times
N g_{n-1} (k')  \tilde g_{n-1}(k) 
\frac{1}{(N\theta)^{n-1/2}} 
\rho_j (N\theta) \chi_2(N \theta) \,
\label{Bessel}
\end{multline}
and
\begin{multline} \notag
\Upsilon_3(\omega,\theta)
= 
\sum_ {k,k'\ge 1 }
e^{i[tkk' +\omega (k'-k)]} \varphi
(h^2 kk') \tilde\psi(k,k') \eta_+(k'-k) \\
\times  N g_{n-1} (k')  \tilde g_{n-1}(k) 
\frac{1}{(N\theta)^{n-1}}  \chi_2(N
\theta)  \mathcal O \big((N\theta)^{-3/2}\big)\,.
\end{multline}
\begin{lemma}\label{symbol-Sigmap21}
For $N=k'+k$,  and $\beta=k'-k$, set
$$
\sigma_2(k,k') =
 \tilde\psi(k,k') \eta_+(k'-k) 
\frac{1}{(N\theta)^{n-\frac12}}
\chi_2(N\theta) R (N\theta)\, ,
$$
where 
\begin{equation}\label{definizioneRe}
R (N\theta)=
\begin{cases}
\rho_j (N\theta)& \text{  for $\Upsilon_j$\,, \,$j=1,2\,,$}\\
(N\theta)^{1/2}\,
\mathcal O \big( (N\theta)^{-3/2}\big) &\text{ for $\Upsilon_3$\,.}
\end{cases}
\end{equation}

Then $\sigma_2$ is a symbol of order 0 in $k'$ ($k$ resp.), depending on the
parameters $\theta$ and $k$ ($k'$ resp.), with norm uniformly bounded in such
parameters. 
\end{lemma}

\begin{proof}
As in Lemma \ref{symbol-Sigmap1}
we wish to show that, considering $k'=\xi$ as a continuous parameter,
$\sigma_2$ is a smooth function of $\xi$ and, for each non-negative
integer $m$ there exists a positive constant $C=C_m$, independent of
$k$ and $\theta\in (0,\varepsilon_1)$, such that
$$
\big|\partial_\xi^m \sigma_2(\xi)\big|
\le C (1+|\xi|)^{-m}\, ,
$$
and we may assume that we have extended
$\sigma_2$ to be identically $0$ when $\xi\le1/2$.

Since
$\chi_2(N\theta)=0$ for $N\theta\le1$
it follows that  $\sigma_2$ is smooth and
bounded as $\xi\to+\infty$, uniformly in $\theta$ and $k$.

Next we consider the derivatives.  
It the derivative falls on the factor
$\frac{ 1}{(N\theta)^{n-\frac12}} $ (or on 
$\frac{ 1}{(N\theta)^{n-1}} $ in the case of $\Upsilon_\mathcal R$), 
using the condition
$N\theta\ge1$
 we easily obtain,respectively,
that
$$
\Big| \partial_\xi 
\Big[ \frac{1}{(N\theta)^{n- \frac12}} \Big]\chi_2(N\theta)\tilde\psi(k,k')\eta_+(k'-k)\Big| 
\le C  \frac1N \le C \frac{1}{\xi}\, ,
$$
and
$$
\Big| \partial_\xi 
\Big[ \frac{1}{(N\theta)^{n- 1}} \Big]\chi_2(N\theta)\tilde\psi(k,k')\eta_+(k'-k)R (N\theta)\Big| 
\le C  \frac1N \le C \frac{1}{\xi}\, ,
$$
as $\xi\to+\infty$, uniformly in $\theta$.

If the derivative falls on $\rho_j (N\theta)$, by means of formula (15) in 
\cite[p.~338]{Stein}
we observe that
$$\Big| \partial_\xi 
\Big[ \sum_{k}a_{k,j}  (\xi \theta)^{-k} 
 \Big]
\tilde\psi(k,k')\eta_+(k'-k)
\frac{1}{(N\theta)^{n-\frac12}}
\chi_2(N\theta) 
\Big| 
 \le C \frac{1}{\xi}\, ,
$$
where $a_{k,j}$, $j=1,2$, are suitable coefficients. 

If the derivative falls on $\chi_2$, we notice that
$\chi_2'(\xi)=0$ unless $1\le\xi\le2$, so that
$$
\Big| \frac{1}{(N\theta)^{n-\frac12}} 
\partial_\xi 
\big[\chi_2(N\theta)\big] \tilde\psi(k,k')\eta_+(k'-k)\Big| 
\le C  \theta \le C\frac1N \le C \frac{1}{\xi}\, ,
$$
and the same is true when $R (N\theta)$ is defined as in the latter
case of 
\eqref{definizioneRe}.

Finally, if the derivative falls on the remainder term
$\mathcal O \big( (N\theta)^{-3/2}\big)$
in the latter case of 
\eqref{definizioneRe},
we have
$$
\Big| 
 \frac{ 1}{(N\theta)^{n- 1}} \chi_2(N\theta)\tilde\psi(k,k')\eta_+(k'-k)
 \partial_\xi \Big[\mathcal O \big( (N\theta)^{-3/2}\big)\Big]\Big| 
\le 
C  \frac1N \le C \frac{1}{\xi}\, .
$$

Hence,
$$
\big|\partial_\xi \sigma_2(\xi)\big|
\le C (1+|\xi|)^{-1}\, ,
$$
as $|\xi|\to+\infty$, uniformly in $\theta$.

The argument can be repeated for all higher order derivatives.
In particular, when the derivatives involve the term
$\mathcal O \big( (N\theta)^{-3/2}\big)$, 
we may  use formula $(8)$, p. 334 in \cite{Stein}, proving that 
this term behaves like a symbol of the expected order. Thus the
lemma is proven.
\end{proof}
\medskip

In the case of $\Upsilon^+$ we still need to use the oscillation of
the phase and hence Lemma \ref{Poisson}.  Let $R_I,R_{I\!I}$ be as in
Lemma \ref{phase-fnc}.  Since the phase in this case is
$tkk' +\omega(k'-k)\pm\theta(k+k')$, and $\theta>0$, we write
$\tilde\theta=\pm\theta$ and let $|\tilde\theta|$ vary in
$[1/N,\varepsilon_1]$. 
We then introduce the  space of parameters $(t, \omega, \theta)$
\begin{equation*}
R_\theta := \big\{
(t,\omega,\tilde\theta)  \in [h^2,h^s]\times [0,2\pi) \times (-\varepsilon_1,\varepsilon_1):\,
|\tilde\theta|<tN/M_1 \text{ or } |\tilde\theta|>M_1 Nt \big\}\, ,
\end{equation*}
where $M_1>2(1+M)$ is a large constant.

\begin{lemma}\label{phase-fnc-2}
Let $\theta$ be such that $1\le N|\tilde\theta|$.  
Let $\mathcal V_+$ be defined as in  \eqref{sV-pm}.
For $(k,k')\in \mathcal V_+$, set 
 $\mu_2= (tk+\omega+\tilde\theta)/2\pi$ and
$\mu_2'=(tk'-\omega+\tilde\theta)/2\pi$.
Then, there exist a constant 
$\gamma>0$ and finitely many regions in $R_\theta$, such that for all $(t,\omega,\tilde\theta)$ belonging to one of these regions,
at least one between the following two conditions
\begin{itemize}
\item[(III)]
$\operatorname{dist}(\mu_2,{\mathbf{Z}})\ge \gamma tk$ for all $(k,k')\in\mathcal V_+$,\smallskip
\item[(IV)]
$\operatorname{dist}(\mu_2',{\mathbf{Z}})\ge \gamma tk'$ for all $(k,k')\in\mathcal V_+$
\end{itemize}
holds.
\end{lemma}

\begin{proof}
We begin observing that, if
either 
$|\tilde\theta|<tN/M_1 \text{ or } |\tilde\theta|>M_1 Nt $, then there exists a constant $C>0$ such that
\begin{equation}\label{stima-dal-basso}
|tk+\tilde\theta| \ge Ctk\, \,\,\text{and }
 |tk'+\tilde\theta| \ge Ctk'\, .
\end{equation}
Then
we split the proof in a few cases. 
\subsubsection*{Case 1}
Let $R_{\theta, I}:=\{
(t,\omega,\tilde\theta)  \in R_\theta\,
:\, 0<c_1<\omega<2\pi-c_1\big\}$,
where $c_1>10\varepsilon_1$ is a small constant.
Then (III)
holds  for $(t,\omega, \tilde\theta)\in R_{\theta, I}$
since
$$
 \operatorname{dist}(\mu_2,{\mathbf{Z}}) 
\ge \operatorname{dist}(\omega/2\pi,{\mathbf{Z}})-\frac{tk}{2\pi}-\frac{|\theta|}{2\pi} 
 \ge 
 \frac{c_1}{2\pi}- \frac{tk}{2\pi} - \frac{\varepsilon_1}{2\pi} 
 \ge 
 c_1\big(\frac{1}{2\pi}-\frac{1}{20\pi}\big)- \frac{tk}{2\pi}
  \ge \gamma tk
\,,
$$
for all $k=1\,,\ldots, c'/h$, as a consequence of \eqref{stima-tk}.

Hence, 
possibly replacing $\omega$ by $2\pi-\omega$, we may
assume that $|\omega|\le c_1$, and that
$\operatorname{dist}(\mu_2,{\mathbf{Z}})= |tk+\omega+\tilde\theta|/2\pi$ and that 
$\operatorname{dist}(\mu_2',{\mathbf{Z}})= |tk'-\omega+\tilde\theta|/2\pi$.
Notice that now we may assume that $\tilde\theta<0$ since otherwise
(III) holds on  $R_I$ and (IV) on $R_{I\!I}$, where $R_I,R_{I\!I}$ are
defined in
Lemma \ref{phase-fnc}. 

\subsubsection*{Case 2}
Let $R_{\theta, I\! I}:=\{
(t,\omega,\tilde\theta)  \in R_\theta\,
:\, |\omega|<c_2 t\,\big\}$,
where $c_2>0$ is a small constant.
If $(t,\omega,\tilde\theta)  \in R_{\theta, I\! I}$, then (III) holds 
for all $(k,k')\in\mathcal V_+$, since
$$|tk+\omega+\tilde\theta|\ge |tk+\tilde\theta|-|\omega|\ge
t( Ck-c_2)\ge \gamma t k\,,$$
provided  that $c_2$ is small enough.

\subsubsection*{Case 3}
Next suppose $c_2t\le |\omega|\le c_1$.
If $\omega>0$, then, if $tk+\tilde\theta>0$, 
(III) holds.
If $tk+\tilde\theta<0$, 
we first observe that
$|\tilde\theta|>M_1 tN$.
Indeed, if    
$|\tilde\theta|<tN/M_1$, then
$$|\tilde\theta|<\frac{1}{M_1}t (k+k')\le
\frac{1}{M_1}tk (1+M)\le \frac{tk}{2}\,,$$
since we chose $M_1 >2(1+M)$. Thus $tk+\tilde\theta>0$, contradicting the hypothesis.
\\
Now it is easy to conclude that
condition (IV) holds,
since
 \begin{equation*}
tk'+\tilde\theta<tk'-M_1 tN<0\,,
\end{equation*}
so that
\begin{equation*}
|tk'+\tilde\theta-\omega|=\omega+
|tk'+\tilde\theta|\ge C t k'\,.
\end{equation*}
If $\omega<0$, then, if $tk+\tilde\theta<0$, 
(III) holds as a consequence of \eqref{stima-dal-basso}.
If $tk+\tilde\theta>0$, 
we notice that
$|\tilde\theta|<\frac{1}{M_1} tN$.
Then condition (IV) holds,
since
 \begin{equation*}
tk'+\tilde\theta>
tk'-\frac{1}{M_1} tN>
tk' (1-\frac{1}{M_1})-\frac{M}{M_1}tk'>0
\end{equation*}
provided that $M_1>1+M$, so that
$ 
|tk'+\tilde\theta-\omega|=|\omega|+
tk'+\tilde\theta\ge C t k'$.
\end{proof}

\begin{proposition}\label{estimate-Sigma-pm}
 There exists a constant $C>0$ such that,
  for all $|t|\in A$ when $n>1$,
and for all $h^2\le |t|\le C h^{4/3}$
when $n=1$,
\begin{equation*}
\big| \Upsilon_j(\omega\,,\theta)\big|
\le C \frac{1}{|t|^{n+1}}
\end{equation*}
uniformly in $\omega$ and $\theta\in[0,\varepsilon_1]$, for $j=1,2$.
\end{proposition}

\begin{proof}
Recall that
$\Upsilon_j$
have been  defined in \eqref{Bessel}.  
Next, we fix a smooth cut-off function $\Psi$ with compact support in
$[1/M_1,M_1]$, where  $M_1>1$ is a (large) constant.
For $j=1,2$ we decompose $\Upsilon_j$ by
setting 
\begin{align}
\Upsilon_j(\omega,\theta)
& = \sum_ {k,k'\ge 1 }
e^{i[tkk' +\omega (k'-k)\pm\theta(k+k')]} \varphi
(h^2 kk') \tilde\psi(k,k')\eta_+(k'-k) \chi_2(N
\theta)
\frac{N g_{n-1} (k')  \tilde g_{n-1}(k) }{(N\theta)^{n-1/2}}
\notag \\ 
& \qquad\qquad\qquad\qquad\qquad\qquad\qquad\qquad\qquad\qquad\qquad\quad
\times \rho_j (N\theta)\, \big[ 1- \Psi (\theta /tN)\big]\notag \\
& \quad+ \sum_ {k,k'\ge 1 }
e^{i[tkk' +\omega (k'-k)\pm\theta(k+k')]} \varphi
(h^2 kk') \tilde\psi(k,k')\eta_+(k'-k) \chi_2(N
\theta)
\frac{N g_{n-1} (k')  \tilde g_{n-1}(k) }{(N\theta)^{n-1/2}}
\notag \\ 
& \qquad\qquad\qquad\qquad\qquad\qquad  \qquad\qquad\qquad\qquad\qquad\qquad\qquad
\times \rho_j (N\theta)\Psi (\theta /tN)
\notag \\
& =: \Upsilon_{j,1}(\omega,\theta) +\Upsilon_{j,2}(\omega,\theta)
\, ,
\end{align}
 We also have
\begin{align*}
\Upsilon_3 (\omega,\theta)
& = \sum_ {k,k'\ge 1}
e^{i[tkk' +\omega (k'-k)]} \varphi
(h^2 kk') N (kk')^{n-1} \sigma_2(k,k') \, .
\end{align*}

For $\Upsilon_{j,1}$ 
we argue as in the proof of Proposition \ref{estimate-Sigmap1} and
divide into the cases in which (III) or (IV) hold, respectively. If, for
instance, (III) holds, then as in the proof of Case (I)
in Proposition 
\ref{estimate-Sigmap1}, 
we have
\begin{align*}
|\Upsilon_{j,1} (\omega,\theta)|  & \le \frac{C}{h^{2n-2}}  
\sum_ {k=1 }^{\lfloor c'/h  \rfloor}
 \Big|
\sum_{k'=1 }^{\lfloor c'/h  \rfloor}
 e^{ik'[tk +\omega +\tilde\theta]}
\varphi (h^2 kk') N (h^2 kk')^{n-1}
 \sigma_2(N) \Big|\\
&  \\
& \le \frac{C}{h^{2n}}  \sum_ {k=1 }^{\lfloor c'/h  \rfloor}
\frac{1}{k} \sum_{k'=  \lfloor a'/(h^2 k) \rfloor}^{\lfloor b'/(h^2 k)
  \rfloor+1} 
 e^{ik'[tk +\omega \pm\theta]}
\varphi (h^2 kk')  (h^2 kk')^{n}
 \sigma_2(N) \Big|  \\
& \qquad\qquad+\frac{C}{h^{2n-2}}  \sum_ {k=1 }^{\lfloor c'/h  \rfloor}k
 \sum_{k'=  \lfloor a'/(h^2 k) \rfloor}^{\lfloor b'/(h^2 k)
  \rfloor+1} 
 e^{ik'[tk +\omega \pm\theta]}
\varphi (h^2 kk')  (h^2 kk')^{n-1}
 \sigma_2(N) \Big|  \\
& \le \frac{C}{h^{2n+\kappa}}  \, ,
\end{align*}
for any $\kappa>0$,
uniformly in $\omega$ and  $\theta\in[0,\varepsilon_1]$,
by proceeding as in the proof of 
\eqref{end-case-1} by means of Lemma
\ref{phase-fnc-2}. 
The proof in the case in which condition (IV) holds is analogous to the proof of
\eqref{end-case-2} in Proposition
\ref{estimate-Sigmap1}, 
and it is omitted.\medskip

Finally we estimate $\Upsilon_{j,2}$.
In this sum we take advantage
of the fact that $\theta/tN$ is bounded above and below from zero.
We have
\begin{align*}
|\Upsilon_{j,2}(\omega,\theta)|
& \le C\sum_ {k,k'\ge 1 }
\varphi (h^2 kk') \tilde\psi(k,k')\eta_+(k'-k) \chi_2(N \theta)
\frac{N^{2n-1}}{(N\theta)^{n-1/2}} \Psi (\theta /tN) \\
& \le  C\sum_ {k,k'\ge 1 }
\varphi (h^2 kk') \tilde\psi(k,k')\eta_+(k'-k) 
\frac{N^{2n-1}}{(N^2t)^{n-1/2}} \\
& \le  \frac{C}{t^{n-1/2}} \sum_ {k,k'\ge 1 }
\varphi (h^2 kk') \tilde\psi(k,k') \\
& \le 
 \frac{C}{t^{n-1/2} h^2}
 \le \frac{C}{t^{n+1+1/n-1/2}}
 \,,
\end{align*}
since $t\le h^s$, so that $h^{-2}\le
t^{-2/s}<t^{-(n+1)/n}$.
Thus when $n=1$
\begin{equation*}
|\Upsilon_{j,2}(\omega,\theta)|
\le
 \frac{C}{t^{2}}
\end{equation*}
for all $h^2\le |t|\le c h^{4/3}$, while
if $n>1$
\begin{equation*}
|\Upsilon_{j,2}(\omega,\theta)|
\le
 \frac{C}{t^{n+1}}
\end{equation*}
for all $|t|\in A$.
Finally, we observe that $\Upsilon_3$
may be treated as the sum 
in \eqref{suffices} by means of Lemma \ref{phase-fnc}, so that
$$
\big|\Upsilon_3 (\omega,\theta)\big|
\le C \frac{1}{|t|^{n+1}}\, ,
$$
uniformly in $\omega$ and $\theta\in[0,\varepsilon_1]$, for all $|t|\in [h^2,
ch^s]$.
\end{proof}

\medskip

Thus, as a consequence of the decompositions
\eqref{sommaspezzata}
and
\eqref{decomp-ypsilon}, 
of 
Propositions
\ref{estimate-Sigmap1} and \ref{estimate-Sigma-pm},
by using also \eqref{est-SigmaR}
and
\eqref{est-SigmaR-2},
if $n>1$  in \eqref{Lj} 
we obtain
$$\big|
K_1^\pm(\omega,\theta)+K_2^\pm (\omega,\theta)
\big|\le
 \frac{C}{|t|^{n+1}}\, ,
$$
uniformly in $\omega$ and  $\theta\in[0,\pi/2-\varepsilon_1]$, so that 
 we finally get  \eqref{normasemplificata}, that is,
$$\sup_{(z,w)\in\Omega} \Big|
 \sum_ {{\ell,\ell'=1}}^{+\infty} 
  e^{it{\lambda}_{\ell,\ell'}}\varphi (h^2{\lambda}_{\ell,\ell'} )\psi(\ell'/\ell)Z_{\ell,\ell'} (z,w)\Big|\le C\frac{1}{h^{2n+\kappa}}\, ,
$$
for all $t\in A$, $n>1$.
When $n=1$, as a consequence of Proposition \ref{estimate-Sigma-pm} 
we get
$$\sup_{(z,w)\in\Omega} \Big|
 \sum_ {{\ell,\ell'=1}}^{+\infty} 
  e^{it{\lambda}_{\ell,\ell'}}\varphi (h^2{\lambda}_{\ell,\ell'} )\psi(\ell'/\ell)Z_{\ell,\ell'} (z,w)\Big|\le C\frac{1}{t^{2}}\, 
$$
for all $|t|\in [h^2, ch^{4/3} ]$.

This concludes  the proof of 
Theorem \ref{dispersive}. \qed  

\medskip

\section{The Strichartz estimate}\label{dim_teo}\medskip

In this section we complete the proof of Theorem
\ref{strichartz}.

Following a classical pattern 
we invoke a  result by Keel and Tao \cite{KeelTao}.
Consider
the family of  operators
$$
U(t):=\chi_J (t)\, e^{it\mathcal L}\,\varphi (h^2 \mathcal L)\,,$$
where $t\in{\mathbf{R}}$, $h\in(0,1]$, $\chi_J$ denotes the characteristic
function of the interval
$J$ and $|J|\approx  h^\alpha$, where, if $n>1$, 
we will select either
$\alpha=s>s_n$  
(with $s_n$ given by \eqref{def-sn})
or $\alpha=2$.
If $n=1$,
we will select either
$\alpha\ge 4/3$  or $\alpha=2$.

Then $U(t)$ satisfies the energy estimate
$\|U(t)\|_{(L^2,L^2)}\le C$ for some positive constant $C$ and
the following {\it{untruncated decay estimate}}
\begin{align*}
\|U(t)U(\tau)^*v_0\|_{L^{\infty}}&\|\chi_J (t-\tau)e^{i(t -\tau)\mathcal L} \varphi (h^2 \mathcal L)v_0\|_{L^{\infty}}\\
& \le \frac{C}{|t-\tau|^{Q/2}}\|v_0\|_{L^1 (S^{2n+1})}
\end{align*}
for all $t,\,\tau\in{\mathbf{R}}$, $t\neq \tau$.
Hence  Theorem 1.2
in \cite{KeelTao} yields the following result.

\begin{proposition}
For any fixed  
$\varphi \in  {\mathcal{C}}_0^{\infty}({\mathbf{R}}_+)$,
there exists a constant $C>0$
such that for  all $h\in (0,1]$,
 for any interval $J$ of length $|J|\le h^\alpha$
 and for all $v_0\in {\mathcal{C}}^{\infty}(S^{2n+1})$ 
the following estimate holds
\begin{equation}\label{3.6}
\Big( \int_J \| e^{it\mathcal L}\,\varphi (h^2 \mathcal L)v_0\|_{L^q}^p\,dt\Big)^{1/p}
\le C
\|v_0\|_{L^2}
\end{equation}
for all pairs $(p,q)\neq (2,+\infty)$, 
satisfying \eqref{condammiss}, where $\alpha=2$ if $v_0\in {\mathcal{C}}^\infty_\mathcal E$
and $\alpha>s_n$ if $v_0\in {\mathcal{C}}^\infty_\mathcal V$, and $\mathcal E,\mathcal V$ are defined
in \eqref{def-cono} and \eqref{def-spigolo}.
Here $C$ depends only on $p,q,n$ and $s$. \medskip
\end{proposition}

Finally, in order to prove
the Strichartz estimate \eqref{Strichartz}
we shall need  an easy consequence 
of the Littlewood--Paley decomposition
for the sublaplacian $\mathcal L$
on the complex sphere. 
More precisely, we shall use the following result.
\begin{theorem}\label{littlewood-paley}
Let $\tilde\psi \in{\mathcal{C}}^{\infty}_0 ({\mathbf{R}}_+)$
and $\psi \in {\mathcal{C}}^{\infty}_0 ({\mathbf{R}})$, such that
\begin{equation}
\tilde\psi (\lambda)+\sum_{j=1}^{\infty}\psi (2^{-2j}\lambda )=1\,,\;\lambda\in{\mathbf{R}}\,.
\end{equation}
Then for $2\le q<\infty$ there exists a constant $C_q$
such that
\begin{equation}\label{LittlPaley}
\|f\|_{L^q (S^{2n+1})}\le C_q \Big(
\| \tilde\psi (\mathcal L)f\|_{L^q (S^{2n+1})}+
\Big(
\sum_{j=1}^{+\infty}
\| \psi (2^{-2j}\mathcal L)f\|_{L^q(S^{2n+1})}^2
\Big)^{1/2}
\Big)\, ,
\end{equation}
for  $f\in L^q(S^{2n+1})$.
\end{theorem}

This  result and  the Littlewood--Paley decomposition
are of independent interest (see the recent paper \cite{Bouclet} for a discussion
of analogous inequalities in the Riemannian case),
and are proved
in the forthcoming paper \cite{CaPe2}.

\subsubsection*{End of the proof of Theorem \ref{strichartz}}
Let  $v_0\in {\mathcal{C}}^\infty_\mathcal V$, $v_0\in {\mathcal{C}}^\infty_\mathcal E$ respectively,
and $\alpha=s>s_n$, $\alpha=2$ respectively.

By writing $[-1,1]=\cup_{k=1}^{N}J_k$,
with $J_k$ intervals, $|J_k|\approx h^\alpha $
 and $N\approx h^{-\alpha}$, we have
\begin{align*}
\int_{-1}^{1}
\| e^{it\mathcal L}\,\varphi (h^2 \mathcal L)v_0\|_{L^q}^p\,dt &\le 
\sum_{k=1}^{N}
\int_{J_k}
\| e^{it\mathcal L}\,\varphi (h^2 \mathcal L)v_0\|_{L^q}^p\,dt\\
&\le C N  \|v_0\|_{L^2}^p\\
& \le C  h^{-\alpha}\|v_0\|_{L^2}^p\,,
\end{align*}
 so that
 \begin{equation}\label{stima-1,1}
\Big(
\int_{-1}^1
\| e^{it\mathcal L}\,\varphi (h^2 \mathcal L)v_0\|_{L^q}^p\,dt\Big)^{1/p}\le
Ch^{-\alpha /p} \|v_0\|_{L^2} \,.
\end{equation}
Now, let 
 $\tilde{\varphi}\in {\mathcal{C}}_0^{\infty}({\mathbf{R}}_+)$ be such that 
$\tilde{\varphi}\,\varphi=\varphi$.
Then   \eqref{stima-1,1},
with $\varphi$ replaced by $\tilde\varphi$ and the initial datum
$ \varphi (h^2 \mathcal L)v_0$,
gives
$$
\Big( \int_{-1}^1
\| e^{it\mathcal L}\,\tilde\varphi (h^2 \mathcal L)\, \varphi (h^2 \mathcal L)v_0\|_{L^q}^p\,dt\Big)^{1/p}\le
Ch^{-\alpha /p}\|\varphi (h^2 \mathcal L) v_0\|_{L^2}\,,
$$
that is,
\begin{equation}\label{local-strich}
\Big(
\int_{-1}^1
\| e^{it\mathcal L}\, \varphi (h^2 \mathcal L)v_0\|_{L^q}^p\,dt\Big)^{1/p}\le
Ch^{-\alpha /p}\|\varphi (h^2 \mathcal L) v_0\|_{L^2}
\,.
\end{equation}

We now apply Theorem \ref{littlewood-paley}
to $f=v(t)=e^{it\mathcal L} v_0$ and then we take the $L^p$-norm
with respect to the variable $t$ on $[-1,1]$ and obtain that
$$
\|v \|_{L^p( [-1,1], L^q (S^{2n+1}))}
\le C\Big( \| v_0\|_{L^2 (S^{2n+1})}
+\Big\|
\Big(\sum_{j=1}^{+\infty}
\| e^{it\mathcal L} \psi (2^{-2j}\mathcal L)v_0\|_{L^q(S^{2n+1})}^2
\Big)^{1/2} \Big\|_{L^p ([-1,1])} \Big)
\, .
$$
Next, let $v_0$ be any function in ${\mathcal{C}}^\infty(S^{2n+1})$. 
Since $p\ge2$,  by Minkowski's integral inequality we have
\begin{align*}
& \Big\| \Big(\sum_{j=1}^{+\infty}
\|e^{it\mathcal L} \psi (2^{-2j}\mathcal L)v_0\|_{L^q(S^{2n+1})}^2
\Big)^{1/2} \Big\|_{L^p ([-1,1])} \\
& \quad \le 
\bigg(
\sum_{j=1}^{+\infty}
\Big(
\int_{-1}^1
\|e^{it\mathcal L} \psi (2^{-2j}{\mathcal{L}})v_0\|_{L^q(S^{2n+1})}^{p}
dt \Big)^{2/p}  \bigg)^{ 1/2}\\
& \quad \le 
\bigg(
\sum_{j=1}^{+\infty}
\Big(
\int_{-1}^1
\|e^{it\mathcal L} \psi (2^{-2j}{\mathcal{L}})\pi_\mathcal V v_0\|_{L^q(S^{2n+1})}^{p}
dt \Big)^{2/p}  \bigg)^{ 1/2}
\\ & \quad\quad\quad\quad\quad\quad
+ \bigg(
\sum_{j=1}^{+\infty}
\Big(
\int_{-1}^1
\|e^{it\mathcal L}  \psi (2^{-2j}{\mathcal{L}})\pi_\mathcal E v_0\|_{L^q(S^{2n+1})}^{p}
dt \Big)^{2/p}  \bigg)^{ 1/2}\\
& \quad\le C
\Big(
\sum_{j=1}^{+\infty}
2^{(2js/p)}
\| \psi (2^{-2j}{\mathcal{L}})\pi_\mathcal V v_0\|_{L^2(S^{2n+1})}^{2}
\Big)^{ 1/2}
+ C\Big(
\sum_{j=1}^{+\infty}
2^{(4j/p)}
\| \psi (2^{-2j}{\mathcal{L}})\pi_\mathcal E v_0\|_{L^2(S^{2n+1})}^{2}
\Big)^{ 1/2}
\\
& \quad \le C \| v_0\|_{\mathcal X^{(s/p,2/p)}} \,,
\end{align*}
where we used, in particular, the Strichartz estimate
\eqref{local-strich} for the spectral truncations. 
This yields
\eqref{Strichartz}.
Analogous arguments lead to
\begin{equation*}
 \Big\| \Big(\sum_{j=1}^{+\infty}
\|e^{it\mathcal L}  \psi (2^{-2j}{\mathcal{L}})v_0\|_{L^q(S^{3})}^2
\Big)^{1/2} \Big\|_{L^p ([-1,1])}
 \le C \| v_0\|_{\mathcal X^{(s/p,2/p)}}
\end{equation*}
for all $s\ge 4/3$, when $n=1$.
\qed
\medskip

The following result follows at once from 
 the Strichartz estimates
\eqref{Strichartz}  and Minkowski inequality.
\begin{corollary}\label{Duhamel}
If $p$ and $q$ satisfy
$\frac{2}{p}+\frac{Q}{q}=\frac{Q}{2}$, $p\ge 2$, $q<\infty$, then
for all $T>0$ and for all $s>s_n$, $s_n$ defined by
\eqref{def-sn}, if $n>1$, 
or for all $s\ge 4/3$ if $n=1$,
there exists $C=C(p,T,s)$ such that for every $f\in L^1 ([-T,T], \mathcal X^{s/p,2/p})$
we have
\begin{equation}\label{Minkowski}
\big\| \int_0^t e^{i(t-t'){\mathcal{L}}}
f(t')dt'\big\|_{L^p([-T,T], L^q (S^{2n+1}))}\le C
\|f\|_{L^1([-T, T], \mathcal X^{(s/p,2/p)})}\,.
\end{equation}
\end{corollary}
\smallskip

\section{{Final remarks}}\label{finalremarks}
\medskip

\subsection{\it{Discussion of optimality.}}\label{optimality}
Strichartz estimates proved in Theorem \ref{strichartz}
are, in general,  not sharp. To study optimality, we may use some 
sharp estimates for the joint spectral projections
$\pi_{\ell,\ell'}$, proved by the first author in
\cite{Casarino1, Casarino2}.

More precisely, consider 
an eigenfunction of the
 sublaplacian ${\mathcal{L}}$ corresponding to the eigenvalue
 $N=\lambda_{\ell,\ell'}$,
and then take
the solution of the homogeneous Schr\"odinger equation 
$v(t,z)=  e^{-i t\lambda_{\ell,\ell'}} v_0$, with initial datum
$v_0= h_{\ell,\ell'}$. Here $h_{\ell,\ell'}$ is a spherical harmonic in
 ${\mathcal{H}}^{\ell,\ell'}$, 
 such that 
 \begin{equation}\label{from-below}
\|h_{\ell,\ell'}\|_{L^q(S^{2n+1})}\ge \frac1C
\big( \lambda_{\ell,\ell'}+1\big)^{\alpha(1/q, n)} 
\big(\ell+\ell'+1\big)^{\beta (1/q, n)}
\|h_{\ell,\ell'}\|_{L^2(S^{2n+1})}\, ,
\end{equation}
where
$$\alpha(1/q,n)= 
\begin{cases} n(\frac{1}{2}-\frac{1}{q}) -{\frac{1}{2}} \quad&\text{ if }\quad 
q\ge 2 \frac{2n+1}{2n-1}\cr
{\frac{1}{2}}\bigl(\frac1q-\frac12\bigr)& \text{ if }\quad 
 2\le q\le 2\frac{2n+1}{2n-1}\,.
\end{cases}
$$
and
$$
\beta(1/q,n)= 
\begin{cases} \frac{1}{2} \quad&\text{ if }\quad 
q\ge 2 \frac{2n+1}{2n-1}\cr
(n+\frac12)
\bigl(\frac12-\frac1q\bigr)& \text{ if }\quad 
 2\le q\le 2\frac{2n+1}{2n-1}\,.
\end{cases}
$$
We refer 
to  Theorem 3.1 and Proposition 3.4 in \cite{Casarino2} for the details.

Then we have
 \begin{align*}
\|v\|_{L^p(I, L^{{q}} (S^{2n+1}))}
 &=\Bigg(  \int_I\Big(
 \int_{S^{2n+1}} \Big|  e^{-i t\lambda_{\ell,\ell'}}
 h_{\ell,\ell'}(z)\Big|^q d\sigma(z)\Big)^{p/{q}}dt\Bigg)^{1/p}  \\
 &=  \ell(I)^{1/p} \big\|h_{\ell,\ell'}\big\|_{L^q(S^{2n+1})}\\
 &\ge \frac1C
\big( \lambda_{\ell,\ell'}+1\big)^{\alpha(1/q, n)} 
\big(\ell+\ell'+1\big)^{\beta (1/q, n)}
\|h_{\ell,\ell'}\|_{L^2(S^{2n+1})} \,.\\
 \end{align*}
 Now if $(\ell,\ell')\in \mathcal V$, for some fixed proper cone $\mathcal V$, defined as in \eqref{def-cono},
and
 if $p,q$ satisfy the admissibility condition \eqref{condammiss}, then 
  \begin{align*}
\|v\|_{L^p(I, L^{{q}} (S^{2n+1}))}
 &\ge \frac1C
({\lambda}_{\ell,\ell'}+1)^{\alpha+\beta/2} 
\|h_{\ell,\ell'}\|_{L^2(S^{2n+1})} \approx
\| v_{0}\|_{W^{\alpha+\beta/2 }(S^{2n+1}) }  \, .
 \end{align*}
It is easy to check that  $s/p>s_n /p>\alpha+\beta/2$, so that this estimate
does not provide the sharp bound.

If $(\ell,\ell')\in \mathcal E$, where $\mathcal E$ is defined as in \eqref{def-spigolo},
 then 
  \begin{align*}
\|v\|_{L^p(I, L^{{q}} (S^{2n+1}))}
 &\ge \frac1C
({\lambda}_{\ell,\ell'}+1)^{\alpha+\beta} 
\|h_{\ell,\ell'}\|_{L^2(S^{2n+1})} \\
 &\ge
  \frac1C 
  \big( \lambda_{\ell,\ell'}+1\big)^{\frac{2n}{pQ}   }
\|h_{\ell,\ell'}\|_{L^2(S^{2n+1})} 
\approx
\frac1C \|  v_{0}\|_{W^{4n/{pQ}}(S^{2n+1}) }\,.
  \end{align*}
Now observe that
$$
\frac{4n}{pQ}\frac{2}{p}\big(
1-\frac{1}{n+1}\big)=s_n\,,\qquad\text{for $n>1$\,,}$$
so that
$$
\|v\|_{L^p(I, L^{{q}} (S^{2n+1}))}
\ge \frac1C \| v_0\|_{W^{\frac{2}{p}(1-1/(n+1))} }
\, 
$$
for all $(p,q)$ satisfying
\eqref{condammiss}. 
Anyway, in Theorem
\ref{strichartz}
we proved that, if $(\ell,\ell')\in\mathcal E$, then the critical index is
$2/p$, instead of
$s_n$. In other words,
 the index $s> \frac2p\big[1-\frac{1}{n+1}\big]$ in
Theorem \ref{strichartz} would be sharp, 
up to the loss of $\varepsilon$ derivatives, if we were able to prove an estimate like
\eqref{Strichartz}
with the space 
$\mathcal X^{(s/p, 2/p)}(S^{2n+1})$ replaced by $W^{s/p}$.
\bigskip 

\subsection{\it{Comparison with other subriemannian frameworks.}}
As recalled in the Introduction,
it has been proved in
 \cite{BahouriGerardXu} that 
no (global in time) dispersive estimate may hold for solutions of the
Schr\"odinger equation on ${\mathbf{H}}_n$. 
Anyway, the situation seems to be less rigid on the reduced Heisenberg group
${\mathbf{h}}_n$,  defined as
${\mathbf{h}}_n:{\mathbf{C}}^n\times{\mathbf{T}}$, with product
$$
(z,e^{it})(w,e^{it'}):
=
\big(
z+w,e^{i\left(t+t'+ \Im m\, z\bar{w}\right)}
\big)\,,
$$
with $z,w\in{\mathbf{C}}^{n}$,
$t,s\in{\mathbf{R}}$,
due to the compacteness of the center.
We point out that there is an intimate connection  between 
the reduced Heisenberg group and the unit complex sphere, since
${\mathbf{h}}_n$ turns out to be a contraction
of  $S^{2n+1}$
 (see \cite{CasarinoCiatti} for more details about the construction of this contractive map).
A detailed discussion of (local in time) dispersive estimates and of
Strichartz 
estimates for solutions of the Schr\"odinger equation on ${\mathbf{h}}_n$
requires some additional care and  
will be presented elsewhere.
\bigskip

\subsection{\it{Discussion of the admissibility conditions.}}

Admissibility condition
\eqref{condammiss}
has been directly inspired by the  scale invariance condition
\eqref{condammissriem}
 on a Riemannian compact manifold of dimension $d$
 (which in turn has been inherited 
by the euclidean space ${\mathbf{R}}^d$), 
with the dimension $d$ replaced by the homogeneous dimension $Q$.
Anyway, on the $\operatorname{CR}$ 
 sphere
the notion of dilation, which leads to 
\eqref{condammissriem}
in the euclidean context, is not intrinsic. 
An interesting possibility could  be  
investigating scaling conditions in the subriemannian framework
of the reduced Heisenberg group, where dilations are well defined as
$\lambda\circ (z, t)=(\lambda z_1, \ldots, \lambda z_n, \lambda^2 t)$,
and then 
importing  them on $S^{2n+1}$.
\bigskip


\end{document}